\documentclass[10pt,a4paper,leqno]{amsart}
\usepackage{amssymb, amsmath, amscd}
\usepackage[hypertexnames=false]{hyperref}

\usepackage{color}

\DeclareMathOperator{\Hes}{\rm Hes}
\DeclareMathOperator{\hes}{\rm hes}
\DeclareMathOperator{\Ric}{\rm Ric}
\DeclareMathOperator{\tr}{\rm tr}

\DeclareMathOperator{\spanned}{\rm span}

\newcommand{\nh}{\nabla h}
\newcommand{\KN}{\mathbin{\bigcirc\mspace{-15mu}\wedge\mspace{3mu}}}

\makeatletter
\@namedef{subjclassname@2020}{%
	\textup{2020} Mathematics Subject Classification}
\makeatother

\newtheorem{theorem}{Theorem}[section]
\newtheorem{lemma}[theorem]{Lemma}

\newtheorem{corollary}[theorem]{Corollary}
\newtheorem{example}[theorem]{Example}

\theoremstyle{definition}

\theoremstyle{remark}
\newtheorem{remark}[theorem]{Remark}

\newcommand\restr[2]{{
  \left.\kern-\nulldelimiterspace 
  #1 
  \vphantom{\big|} 
  \right|_{#2} 
  }}

\begin{document}
\title{The vacuum weighted Einstein field equations }
\author{M. Brozos-V\'azquez, D. Moj\'on-\'Alvarez}
\address{MBV: CITMAga, 15782 Santiago de Compostela, Spain}
\address{\phantom{MBV:}  
	Universidade da Coru\~na, Campus Industrial de Ferrol, Department of Mathematics, 15403 Ferrol, Spain}
\email{miguel.brozos.vazquez@udc.gal}
\address{DMA: CITMAga, 15782 Santiago de Compostela, Spain}
\address{\phantom{DMA:}
	University of Santiago de Compostela,
15782 Santiago de Compostela, Spain}
\email{diego.mojon@rai.usc.es}
\thanks{Research partially supported by grants PID2022-138988NB-I00 funded by MICIU/AEI/10.13039/501100011033 and by ERDF, EU, and ED431F 2020/04 (Xunta de Galicia, Spain); and by contract FPU21/01519 (Ministry of Universities, Spain).}
\date{}

\maketitle

\begin{abstract} 
On a spacetime $(M,g)$ endowed with a density function $h$, we consider the vacuum weighted Einstein field equations: 
\[h\rho-\Hes_h+\Delta h g=0.\]
First, it is shown that the equation characterizes critical metrics for an appropriate action. Then, after describing locally conformally flat solutions in arbitrary dimension, four-dimensional solutions with harmonic curvature are classified.
\end{abstract}


\section{Introduction}

A usual spacetime is generalized by introducing a density function $h$ which alters the Riemannian volume element $dvol_g$. Thus, we consider a triple $(M,g,h\, dvol_g)$ (or $(M,g,h)$ for short), where $(M,g)$ is a Lorentzian manifold and $h$ is a positive smooth function on $M$. We assume the density to be non-constant on any open subset of $M$ and refer to such a triple as a smooth metric measure spacetime (SMMS). The presence of the density gives rise to new geometric objects that reflect its influence on the geometry of the spacetime. 

The usual Einstein tensor, defined by $G=\rho-\frac{\tau}{2} g$ (where $\rho$ stands for the Ricci tensor and $\tau$ for the scalar curvature), has among its properties being symmetric, divergence-free, concomitant of the metric tensor $g$ and its first two derivatives and linear in the second derivatives of $g$. Furthermore, these features characterize a tensor of the form $G+\Lambda g$, where $\Lambda$ is a constant (see, for example, \cite{Lovelock}). Following this pattern, the weighted Einstein tensor
\[
G^h=h\rho-\Hes_h+\Delta h g
\]
was introduced in \cite{Brozos-Mojon} as a tensor which is symmetric, divergence-free if the scalar curvature is constant, concomitant of the metric, the density function and their first two derivatives, and linear on the second derivatives of $g$ as well as the density function. Here, $\Hes_h$ and $\Delta h$ denote the Hessian tensor and the Laplacian of $h$, respectively. From the previous definition, vacuum weighted Einstein field equations  on a spacetime with density are regarded as $G^h=0$. 

Besides this motivation for equation $G^h=0$, it appears as a very natural second order equation with geometric interest. Indeed, this equation was previously considered in Riemannian signature, both from a purely geometric perspective and for reasons coming from General Relativity on static spacetimes. Although the equation itself is related to works such as \cite{Bla,Lic2,Marsden} through the variation of the scalar curvature, as we shall see below, it was from the work of Kobayashi and Lafontaine \cite{Kobayashi, Lafontaine} that its solutions were studied more systematically in the Riemannian setting (see also \cite{Shen,Qing-Yuan}).

From a variational point of view, Einstein spacetimes arise as critical metrics for the  Einstein-Hilbert action $\mathcal{S}=\int_M\tau \, dvol_g$. Although the previous motivation for the definition of the weighted Einstein tensor is based on reasonable properties that suggest its suitability for a context with the presence of a density, it lacks the advantages of the variational approach, which would provide a useful action to work with. Thus, in order to introduce this formalism in the weighted setting, it would be desirable to find an appropriate action that recovers the vacuum weighted Einstein field equations  $G^h=0$. Hence, we complete the motivation of the definition of $G^h$ by  introducing this new perspective as follows.

\subsection{The weighted Einstein field equation: variational approach}\label{sect:variational-approach}
Consider a smooth manifold $M$ and the space $\mathcal{M}$ of Lorentzian metric measure structures on $M$, i.e. $\mathcal{M}=\text{Lor}(M)\times \mathcal{C}^\infty(M;\mathbb{R}^+)$, where $\text{Lor}(M)$ is the space of Lorentzian metrics on $M$ and $\mathcal{C}^\infty(M;\mathbb{R}^+)$ is the space of positive smooth functions on $M$ playing the role of a density. Assume along this section that the density function has compact support, so that all integrands do. Now, define the {\it weighted Einstein-Hilbert functional} as the functional $\mathcal{S}:\mathcal{M}\rightarrow \mathbb{R}$ given by
\begin{equation}\label{eq:HE-action-weighted}
	(g,h)\mapsto \mathcal{S}_{(g,h)}=\int_M\tau h\, dvol_g,
\end{equation}
where $dvol_g=\sqrt{|g|}\, dx_1\wedge\cdots\wedge dx_n$ is the usual Riemannian volume element, with $|g|=-\det g_{ij}$. For simplicity, we will assume the dependence on the metric and we will write $\mathcal{S}_{(g,h)}=\mathcal{S}_h$ in order to highlight the fact that we work in a weighted setting. Note that, for $h\equiv 1$, we recover the Einstein-Hilbert functional: $\mathcal{S}_1=S$. Nevertheless, we are working with SMMSs given by densities $h$ which are non-constant on any open subset of $M$, so the behavior of $\mathcal{S}_h$ is different from its unweighted counterpart.

A defining feature of a SMMS $(M,g,h)$ as a weighted geometric object is its weighted volume element $dV_h=h\, dvol_g$. Hence, it is natural to pose the variational problem of finding the critical points of $\mathcal{S}_h$, constrained to variations of the metric-measure structure of $(M,g,h)$ which keep the weighted volume element invariant at each point of the manifold. By a variation of the metric-measure structure, we mean a simultaneous variation of both the metric and the density function, with the same variation parameter $t$:
\begin{equation}\label{eq:met-meas-var}
	g[t]=g+t\bar g, \qquad h[t]=h+t\bar h, \qquad dV_h[t]=h[t]dvol_{g[t]}.
\end{equation}
In order to maintain the compact support of the integrands, we shall assume that any variations of $h$ have compact support as well. Moreover, in case of the presence of a boundary, the variations and their first derivatives must vanish on it (see, for example, \cite{Quiros}).



For a variation of the metric-measure structure \eqref{eq:met-meas-var}, through the well-known expression $\delta\sqrt{|g|}=-\frac{1}2\sqrt{|g|}\tr \bar g$, the invariance of the weighted volume element reads
	\[
	\begin{array}{rcl}
		\delta dV_h=\restr{\frac{d}{dt}}{t=0}dV_h[t]&=&\bar h\, dvol_g+h\restr{\frac{d}{dt}}{t=0}dvol_{g[t]}\\
		\noalign{\medskip}
		&=&\left(\bar h-\frac{1}{2}h\,  \mathrm{tr}\,\bar g\right) \, dvol_g=0,
	\end{array}
	\]
	so we have the pointwise condition $\left(\bar h-\frac{1}{2}h\,  \mathrm{tr}\,\bar g\right)=0$. Also due to this invariance, we compute the variation of the functional,
	\[
	\delta \mathcal{S}_h=\restr{\frac{d}{dt}\mathcal{S}_h}{t=0}=\int_M \restr{\frac{d\tau}{dt}}{t=0}hdvol_g+\int_M \tau  \delta dV_h=\int_M \restr{\frac{d\tau}{dt}}{t=0}hdvol_g.
	\]
The last integrand is the linearization of the scalar curvature $D\tau=\restr{\frac{d\tau}{dt}}{t=0}=-\Delta \operatorname{tr}\{\bar g\}+\operatorname{div} \operatorname{div} \bar g-\bar g\cdot\Ric$ (see \cite{Bla,Lic2}), where $\operatorname{div}$ stands for the usual divergence.  Using the formal adjoint of $D\tau$ with respect to the $L^2$-inner product, we can write
\[
\delta \mathcal{S}_h=\int_M \restr{\frac{d\tau}{dt}}{t=0}hdvol_g=\int_M D\tau (\bar g) hdvol_g=\int_M \langle D\tau^\star (h),\bar g\rangle dvol_g,
\] 
where $D\tau^\star(h)=-h  \rho+\Hes_h-\Delta h g$ is derived from the formal adjoints of $\operatorname{div} $ and $\Delta$. Hence
\[
\delta \mathcal{S}_h=\int_M \langle -h  \rho+\Hes_h-\Delta h g,\bar g\rangle dvol_g.
\]
Note that we have the freedom to choose any variation of the metric $\bar g$, since there always exists a variation of the density, given by $\bar h=\frac{1}{2}h\,  \mathrm{tr}\,\bar g$, which preserves the weighted volume element. Hence, the left-hand side of the product $\langle h\rho -\Hes_h+\Delta h g,\bar g\rangle$ must vanish for all critical points of $\mathcal{S}_h$, restricted to variations of the metric-measure structure preserving the weighted volume element. It follows that these critical points satisfy the equation
\begin{equation}\label{eq:vacuum-Einstein-field-equations}
		h\rho -\Hes_h+\Delta h g=0,
\end{equation}
which corresponds to the vacuum weighted Einstein field equations  $G^h=0$.

\subsection{Main results}
The general objective of this paper is to further understand the {\it solutions} of the vacuum weighted Einstein field equations, i.e., the SMMSs $(M,g,h)$ such that \eqref{eq:vacuum-Einstein-field-equations} is satisfied on $M$ for the metric-measure structure $(g,h)$. Thus, in Section~\ref{sect:locally-conf-flat} and beyond, we consider different kinds of solutions and study their geometric features.  Our analysis is local in nature and, in the following results, we assume that $\nh\neq 0$ and that the causal character of $\nh$ is constant in $M$. Hence, for general solutions, results should be applied to suitable open sets where these conditions are met. 

Solutions can be split into two families, according to the causal character of $\nh$: {\it isotropic} solutions are those with $\nh$ lightlike, while {\it non-isotropic} solutions are those with $\nh$ timelike or spacelike. The approach in treating an equation like \eqref{eq:vacuum-Einstein-field-equations} is different for the isotropic and non-isotropic cases, as are often the geometric features of the resulting solutions.

A first approach to isotropic solutions was made in \cite{Brozos-Mojon}, finding that all of them have nilpotent Ricci operator. Moreover, isotropic solutions are Kundt spacetimes, and if the Ricci operator vanishes or is two-step nilpotent, then the underlying Lorentzian manifold is a Brinkmann wave. In this work, we aim to extend this study to non-isotropic solutions, while also further understanding the isotropic case. Non-isotropic solutions present a less rigid geometric structure than their isotropic counterparts, and thus fixing their geometry becomes, in general, an unfeasible task. Since equation~\eqref{eq:vacuum-Einstein-field-equations} provides direct information on the Ricci tensor, it is natural to impose geometric conditions related to the Weyl tensor. First, we go over the locally conformally flat case, which turns out to be quite restrictive, as illustrated by the following theorem. 


\begin{theorem}\label{th:loc-conf-flat-ndim}
Let $(M,g,h)$ be an $n$-dimensional locally conformally flat smooth metric measure spacetime. Then, $(M,g,h)$ is a solution of the vacuum weighted Einstein field equations~\eqref{eq:vacuum-Einstein-field-equations} if and only if one of the following is satisfied: 
\begin{enumerate}
	\item  $g(\nh,\nh)\neq 0$ and $(M,g)$  is locally isometric to a warped product $(I\times N,dt^2\oplus \varphi^2 g^N)$, where $I\subset \mathbb{R}$ is an open interval, $N$ is an $(n-1)$-dimensional manifold of constant sectional curvature, and $h(t)$ and $\varphi(t)$ satisfy the following system of ODEs:
	\begin{equation}\label{eq:weighted-field-ODEs}
		\begin{array}{rcl}
			0&=&h'\varphi'-h\varphi'', \\
			\noalign{\medskip}
			0&=&h''+(n-1)h \frac{\varphi''}{\varphi}+ \varepsilon \frac{\tau}{n-1}h.\\
			\noalign{\medskip}
		\end{array}
	\end{equation}
	
\item $g(\nh,\nh)=0$ and $(M,g,h)$ is a plane wave. Moreover, there exist local coordinates $\{u,v,x_1,\dots, x_n\}$ such that the metric is given by
\[
g(u,v,x_1,\dots,x_{n-2})= 2 dv du+F(v,x_1,\dots, x_{n-2}) dv^2+\sum_{i=1}^{n-2} dx_i^2,
\]
where $F(v,x_1,\dots, x_{n-2})=-\frac{h''(v)}{ (n-2) h(v)} \sum_{i=1}^{n-2} x_i^2+\sum_{i=1}^{n-2} b_i(v) x_i+c(v)$. 
\end{enumerate}
\end{theorem}

The core of this work is the analysis of the less rigid condition $\operatorname{div} W=0$ (i.e., the Weyl tensor is harmonic), for 4-dimensional solutions. Since solutions have constant scalar curvature (see Lemma~\ref{le:const-sc}), $\operatorname{div} W=0$ is equivalent to the harmonicity of the curvature tensor. This condition allows for more geometric flexibility than local conformal flatness, and solutions with different Jordan forms of the Ricci operator arise. Each of these forms requires a different approach, but some geometric features are common to all solutions. Indeed, the Ricci eigenvalues are real (vanishing in the isotropic case), and all solutions with non-diagonalizable Ricci operator are realized on Kundt spacetimes. The following result provides a more detailed description of the underlying Lorentzian manifold.

\begin{theorem}\label{th:classification-harmonic}
Let $(M,g,h)$ be a $4$-dimensional solution of the vacuum weighted Einstein field equations~\eqref{eq:vacuum-Einstein-field-equations} such that $(M,g)$ has harmonic curvature tensor (not locally conformally flat). Assume that the Jordan normal form of the Ricci operator $\Ric$ is constant in $M$. Then, the eigenvalues of $\Ric$ are real and one of the following is satisfied:
\begin{enumerate}
\item  $\Ric$ diagonalizes on $(M,g)$ and $g(\nh,\nh)\neq 0$. Furthermore, there exists an open dense subset $M_{\Ric}$ of $M$ such that, for every $p\in M_{\Ric}$, $(M,g)$ is isometric on a neighborhood of $p$ to:
	\begin{enumerate}
		\item A direct product $I_2\times \tilde M$, where $\tilde M=I_1\times_\xi N$ is a warped product $3$-dimensional solution with $\tilde \tau=0$ and $N$ a surface of constant Gauss curvature.
				
		\item A direct product $N_1\times N_2$ of two surfaces of constant Gauss curvature $\frac{\kappa}{2}$ and $\kappa$, respectively.
	\end{enumerate}
		
	\item $(M,g)$ is a Kundt spacetime and, depending on the causal character of $\nh$, one of the following applies:
	\begin{enumerate}
		\item If $g(\nh,\nh)=0$, then $\Ric$ is nilpotent and $\nh$ determines the lightlike parallel line field. Moreover, if $\Ric$ vanishes or is 2-step nilpotent, the underlying manifold is a $pp$-wave. 
		\item If $g(\nh,\nh)\neq0$, then $\nh$ is spacelike and the distinguished lightlike vector field is orthogonal to $\nh$.
	\end{enumerate}
\end{enumerate}
\end{theorem}

A remarkable observation from this classification is that the underlying spacetimes for the solutions that are obtained are typical examples that also arise in the study of cosmological models in General Relativity without the presence of a density.
The descriptions of the density functions corresponding to each of the manifolds given in Theorem~\ref{th:classification-harmonic} are  discussed in Sections~\ref{sec:4-dim-diagonal} to \ref{sec:non-diagonalizable-real}.

\subsection {Outline of the paper}
The remaining of the paper is organized as follows. Some preliminaries are given in Section~\ref{sect2}, where we review general aspects of manifolds that will appear in subsequent sections (including Kundt spacetimes and warped structures). In Section~\ref{sect:locally-conf-flat} we consider solutions with vanishing augmented Cotton tensor, which leads to the proof of Theorem~\ref{th:loc-conf-flat-ndim}. 
In order to prove Theorem~\ref{th:classification-harmonic}, each admissible form of $\Ric$ is tackled in the corresponding section, in both the isotropic and non-isotropic cases. Further details on the geometry and the form of the density function for solutions with different Jordan forms are also provided in each of them. In particular, in Section~\ref{sec:4-dim-diagonal} we consider the diagonalizable case, where we follow some ideas already used in the Riemannian case (the more technical part is left to an Appendix). In Section~\ref{sec:non-diagonalizable-complex}, we prove that the Ricci eigenvalues of solutions with harmonic curvature are necessarily real. This is a long proof that requires a detailed analysis of the geometry of solutions to end up with the use of an algebraic tool (Gr\"obner bases) on a set of polynomials to show that solutions with non-real eigenvalues do not exist. Finally, in Section~\ref{sec:non-diagonalizable-real} we study non-diagonalizable solutions with minimal polynomial of degrees two and three. In Section~\ref{sec:conclusions}, some conclusions summarize the given results.

\section{Distinguished Lorentzian metric structures}\label{sect2}

Several well-known families of spacetimes appear as underlying manifolds in the analysis of solutions to equation~\eqref{eq:vacuum-Einstein-field-equations}. In this section we recall them and fix notation for the subsequent study.

\subsection{Warped and multiply warped  product metrics}\label{subsec:warped-metrics}

A simple geometric structure is that of a direct product manifold,
this is, a manifold $\mathcal{M}=(M,g)$ which decomposes as
$\mathcal{M}=(B\times F,g_B\oplus g_F)$. If we  modify
the metric of the second factor by a positive warping function defined on the base $B$, we have a {\it warped product} $B\times_f F$ with metric $g_B\oplus f^2 g_F$, where $F$ is the fiber. We will make use of the following expression for the Levi-Civita connection for this structure, for $X,V$ lifts of vectors in $B$ and $F$ respectively:
\begin{equation}\label{eq:connection-warped-product}
	\nabla_XV=\nabla_VX=\frac{1}f X(f)V.
\end{equation}
Additionally, the a priori non-vanishing components of the Ricci tensor for such a structure are given by
\begin{equation}\label{eq:ricci-warped-product}
\begin{array}{rcl}
\rho(X,Y)&=&\left(\rho^B-\frac{d}{f} \Hes_f\right)(X,Y), \text{ where  $X, Y$ are tangent to $B$,}\\
\noalign{\medskip}
\rho(U,V)&=&\left(\rho^F-\left(\frac{f \Delta^B
	f+(d-1)\|\nabla f\|^2}{f^2}\right)g\right)(U,V), \text{ where   $U,V$ are tangent to $F$},
\end{array}
\end{equation}
where $d=\dim F$ (see \cite{Oneill}). 
A useful remark from \cite{Ponge-Twisted} is that a metric $g$ defined on $B\times F$, with foliations $\mathfrak{L}_B$ and $\mathfrak{L}_F$ intersecting orthogonally,
 is a warped product if and only if $\mathfrak{L}_B$ is
totally geodesic and $\mathfrak{L}_F$ is spherical.

If we add several fibers $(F_1,g_1)$, $\dots$, $(F_k,g_k)$, with their corresponding warping functions $f_1$, \dots, $f_k$ defined on $B$, we get a multiply warped product $B\times_{f_1} F_1
\times \dots \times_{f_k} F_k$ with the metric given by $g_B\oplus f_1^2 g_1 \oplus \dots \oplus f_k^2 g_k$.  The expressions for the components of the Levi-Civita connection and the Ricci tensor for this structure can be seen, for example, in \cite{Dobarro-Unal}.

\subsection{Kundt spacetimes, Brinkmann waves and other  subfamilies}\label{subsec:Kundt-pr}

Kundt spacetimes are a broad family of Lorentzian manifolds that appear in a number of different physical and geometric situations. They are characterized by the presence of a distinguished lightlike vector field and stronger features of this vector field result in more specific subfamilies of spacetimes such as Brinkmann waves or $pp$-waves. Literature devoted to these spacetimes is vast and we refer to \cite{Boucetta,chow2010kundt,Kundt-spacetimes,podolsky2009general,stephani-et-al} to cite just a few.

Kundt spacetimes and their subfamilies play an essential role in the study of isotropic solutions of the weighted Einstein field equations . Isotropic solutions have nilpotent Ricci operator, $\Ric$. Moreover, they are Kundt spacetimes if $\Ric$ is $3$-step nilpotent and Brinkmann waves if $\Ric$ is $2$-step nilpotent (see \cite{Brozos-Mojon} for details). 
We will see in the subsequent sections that Kundt spacetimes not only play a role in the isotropic case, but also in the non-isotropic case, under some geometric conditions.

{\it Kundt spacetimes} are Lorentzian manifolds of dimension $n\geq 3$ with a lightlike geodesic vector field $V$ which is recurrent in its orthogonal complement (see, for example, \cite{Boucetta}), i.e. there exists a differential 1-form $\omega$ such that
\begin{equation}\label{eq:condition-K1}
	\nabla_VV=0 \quad \text{and} \quad \nabla_XV=\omega(X)V \,\, \text{ for all }  X\in V^\perp.
\end{equation} 
Alternatively (see
\cite{chow2010kundt,Kundt-spacetimes,podolsky2009general})), the distinguished vector fields $V$ that define Kundt spacetimes are characterized by being lightlike geodesic  and having zero {\it optical scalars}: {\it expansion} ($\theta=\frac{1}{n-2}\nabla_iV^{i}$), {\it shear} ($\sigma^2=(\nabla^i V^j) \nabla_{(i} V_{j)}-(n-2)\theta^2$) and {\it twist} ($\omega^2=(\nabla^iV^j)\nabla_{[i}V_{j]}$,
where parentheses and brackets in the subindices denote symmetrization and anti-symmetrization, respectively). 
For an $n$-dimensional Kundt spacetime, in appropriate local coordinates $(u,v,x_1,\dots, x_{n-2})$, the metric can be written as 
\begin{equation}\label{eq:local-coord-kundt-xeral}
	g=dv\left(2du+F(u,v,x)dv+\sum_{i=1}^{n-2}2 H_{x_i}(u,v,x)dx_i\right)+\sum_{i,j=1}^{n-2}g_{ij}(v,x)dx_idx_j,
\end{equation}
where $F$, $H_{x_i}$ and $g_{ij}$ are functions of the specified coordinates.  
In dimension four, coordinates in \eqref{eq:local-coord-kundt-xeral} can be further specialized so that $g_{ij}=P(v,x)\delta_{ij}$  (see \cite{stephani-et-al}).

Some well-known families of Kundt spacetimes are characterized by more restrictive conditions on the distinguished lightlike vector field $V$. If $V$ is recurrent, i.e. if $\nabla_XV=\omega(X)V$  for a $1$-form $\omega$ on $M$ and every $X\in  \mathfrak{X}(M)$, then the manifold is said to be a {\it Brinkmann wave}. In this case, coordinates \eqref{eq:local-coord-kundt-xeral} can be
specialized with $H_{x_i}$ not depending on $u$ (see, for example \cite{leistner}). Moreover, $\partial_u F=0$ if and only if $V$ can be rescaled to a parallel vector field, in which case the metric can be further simplified.

Among Brinkmann waves, there are special families of interest that are obtained by imposing some conditions on the curvature. Following terminology in \cite{leistner}, {\it pure radiation waves} ($pr$-waves for short), are Brinkmann waves whose curvature tensor satisfies $R(V^\perp,V^\perp,-,-)=0$. In this case, the Kundt coordinates \eqref{eq:local-coord-kundt-xeral} reduce to a much simpler form:
\begin{equation}\label{eq:local-coord-pr-xeral}
	g=2dudv+F(u,v,x)dv^2+\sum_{i,j=1}^{n-2}dx_i^2.
\end{equation}
Whenever the distinguished vector field $V$ is parallel, the $pr$-wave is said to be a $pp$-wave, and $F$ can be taken to satisfy $\partial_uF=0$. Note that a Brinkmann wave with parallel lightlike vector field $V$ is a $pp$-wave if and only if $R(V^\perp,V^\perp,\cdot,\cdot)=0$. Moreover, it was shown in \cite{leistner} that a $pr$-wave is a $pp$-wave if and only if it is Ricci isotropic, i.e. $\Ric(X)=0$ for any $X\in V^\perp$.  These spacetimes are very common in special situations described in General Relativity and, in particular, as solutions of the Einstein equations (we refer to \cite{stephani-et-al} for further details). A $pp$-wave with transversally parallel curvature tensor (i.e. such that $\nabla_{V^\perp} R=0$) is called a {\it plane wave}. Again, we refer to \cite{stephani-et-al} for examples of contexts where these spacetimes play a role, which are numerous. In local coordinates, the metric of a plane wave can be given by \eqref{eq:local-coord-pr-xeral} where $F(u,v,x)=\sum_{i,j=1}^na_{ij}(v)x_ix_j$ and the coefficients $a_{ij}$ are smooth functions of $v$. Note that, if the $a_{ij}$ are constants, these metrics correspond to {\it Cahen-Wallach symmetric spaces} \cite{cahen-wallach}.

\section{Solutions to the vacuum weighted Einstein field equations }\label{sect:locally-conf-flat}

From a geometric point of view, it is natural to question how the geometry of a solution of \eqref{eq:vacuum-Einstein-field-equations} is constrained by the relationship between the density and the metric, and to what end the features of both $g$ and $h$ can be inferred from the vacuum weighted field equations. 
The first geometric consequence of these equations, coming from the fact that $G^h$ is divergence-free if they are satisfied, is that the scalar curvature of a solution is necessarily constant (cf. \cite{Bourguignon,Fischer-Marsden}).

\begin{lemma}\label{le:const-sc} {\rm \cite{Brozos-Mojon}}
	If $(M,g,h)$ is a solution of the vacuum weighted Einstein field equations, then its scalar curvature $\tau$ is constant.
\end{lemma}

We introduce some notation as follows. The $(0,4)$-curvature tensor is given by $R(X,Y,Z,U)=g((\nabla_{[X,Y]}-[\nabla_X,\nabla_Y])Z,U)$. For an $n$-dimensional spacetime the {\it Schouten tensor} is given by $P=\frac{1}{n-2}\left(\rho-\frac{\tau}{2(n-1)}g\right)$, so the Weyl tensor $W$ is obtained from the relation $R=P\KN g+W$, where $\KN$ is the {\it Kulkarni-Nomizu product}.

We use 
 $\operatorname{div}$ to denote the usual divergence, given in an orthonormal frame $\{E_1,\dots,E_n\}$, with $\varepsilon_i=g(E_i,E_i)$, by $\operatorname{div} T(\cdots)=\sum_{i=1}^n\varepsilon_i(\nabla_{E_i}T)(E_i,\cdots)$. Thus, the divergence of the Weyl tensor is $\operatorname{div} W=\frac{n-3}{n-2} dP$, where $dP$ is the {\it Cotton tensor}, given by 
\begin{equation}\label{eq:cotton}
 dP(X,Y,Z)=(n-2)\left\{(\nabla_YP)(X,Z)-(\nabla_ZP)(X,Y)\right\}.
\end{equation} 
Notice that a manifold with $n\geq 4$ is locally conformally flat if and only if $W=0$, so $dP=0$ if this condition holds. In dimension three, besides, $dP=0$ characterizes local conformal flatness.
 

Moreover, the divergence of the Riemann curvature tensor is $\operatorname{div} R(X,Y,Z)=(\nabla_Y\rho)(X,Z)-(\nabla_Z\rho)(X,Y)$. Hence, since the scalar curvature of any vacuum solution is constant by Lemma~\ref{le:const-sc}, the Cotton tensor satisfies $dP=\operatorname{div} R$. Thus, considering solutions with harmonic Weyl tensor and harmonic curvature tensor becomes equivalent.

Let $J=\frac{\tau}{2(n-1)}$ be the usual Schouten scalar. Taking traces in \eqref{eq:vacuum-Einstein-field-equations}, we have $\Delta h=-\frac{h\tau}{n-1}=-2J h$, so equation \eqref{eq:vacuum-Einstein-field-equations} can also be written as
\begin{equation}\label{eq:vacuum-Einstein-field-equations-2}
	h\left(\rho-2J g\right)=\Hes_h,
\end{equation}
and $d\Delta h=-2J dh$ since $\tau$ is constant by Lemma~\ref{le:const-sc}. On the other hand, we have
\[
(\nabla_Z\Hes_h)(X,Y)-(\nabla_Y\Hes_h)(X,Z)=R(\nh,X,Y,Z),
\]
so, using \eqref{eq:vacuum-Einstein-field-equations-2}, we can write
\begin{equation}\label{eq:Rnf}
	\begin{array}{rcl}
		R(\nh,X,Y,Z)&=&\nabla_Z(h(\rho-2J g))(X,Y)-\nabla_Y(h(\rho-2J g))(X,Z)\\
		\noalign{\smallskip}
		&=& \left(\rho-2Jg\right)\wedge dh(X,Y,Z)-h\operatorname{div} R(X,Y,Z)\\
		\noalign{\smallskip}
		&=&(\left(\rho-2Jg\right)\wedge dh-hdP)(X,Y,Z),
	\end{array}
\end{equation}
where, for a $(0,2)$-tensor $T$ and a 1-form $\omega$, $T\wedge \omega(X,Y,Z)=T(X,Y)\omega(Z)-T(X,Z)\omega(Y)$.

Since the spacetime is endowed with a density, following terminology in \cite{Qing-Yuan}, we define the {\it augmented Cotton tensor}
\begin{equation}\label{eq:tensorD1}
	D=hdP+\iota_{\nh}W,
\end{equation}
where $\iota_{\nh}W(X,Y,Z)=W(\nh,X,Y,Z)$. The tensor $D$ is related to the Bach tensor in the direction of $\nh$ and restrictions on it have consequences on the geometry of solutions in Riemannian signature (see \cite{Qing-Yuan} for details). Moreover, the weighted Einstein field equations  gives $D$ a useful alternative characterization.

\begin{lemma}\label{le:nh-eigenvector}
	For any solution of the  vacuum weighted Einstein equations, the augmented Cotton tensor $D$ satisfies
	\begin{equation}\label{eq:tensorD2}
		(n-2)D=(n-1) \rho\wedge dh+ g\wedge \iota_{\nh}\rho-\tau  g\wedge dh
	\end{equation}
		for all $X,Y,Z\in \mathfrak{X}(M)$.
	\end{lemma}
	\begin{proof}
		Substituting the definition of $D$ in \eqref{eq:Rnf}, and using the curvature decomposition $R=P\KN g+W$, we get
		\[
		\iota_{\nh}(P\KN g)=\rho\wedge dh-\frac{\tau}{n-1} g\wedge dh-D.
		\]
		Moreover, by the definition of $P$ and the Kulkarni-Nomizu product,
		\[
		\begin{array}{rcl}
			\iota_{\nh}(P\KN g)&=& -\left(P\wedge dh+g\wedge \iota_{\nh}P\right) \\
			\noalign{\medskip}
			&=&-\frac{1}{n-2}\left(g\wedge\iota_{\nh}\rho+\rho\wedge dh-\frac{\tau}{n-1}g\wedge dh\right).
		\end{array}
		\]
		Equating both expressions for $\iota_{\nh}(P\KN g)$, the result follows.
	\end{proof}

\subsection{Solutions with vanishing augmented Cotton tensor}
As a first step in understanding solutions to the vacuum weighted Einstein field equations, we consider in this subsection those with vanishing augmented Cotton tensor.


We distinguish between isotropic and non-isotropic solutions. Recall that we are assuming that the character of $\nh$ does not change in $M$. As concerns isotropic solutions, they were described in general in \cite{Brozos-Mojon}, showing that the Ricci operator is nilpotent and that solutions are realized in one of the two following families:
\begin{enumerate}
	\item Brinkmann waves if the Ricci operator vanishes or is $2$-step nilpotent.
	\item Kundt spacetimes if the Ricci operator is $3$-step nilpotent.
\end{enumerate}
The following result shows that, if the augmented Cotton tensor vanishes, the underlying manifold is a warped product or a Brinkmann wave, depending on whether it is non-isotropic or isotropic, respectively.
	
\begin{theorem}\label{th:loc-conf-flat-nisolutions}
Let $(M,g,h)$ be a solution of the weighted Einstein field equations \eqref{eq:vacuum-Einstein-field-equations} with vanishing $D$ tensor. 
\begin{enumerate}
	\item If $g(\nh,\nh)\neq 0$, then $(M,g)$ is locally isometric to a warped product $I\times_\varphi N$, where $I\subset \mathbb{R}$ is an open interval, $N$ is an $(n-1)$-dimensional Einstein manifold, and $\nh$ is tangent to $I$. 
	\item If $g(\nh,\nh)=0$, then $(M,g)$ is a Brinkmann wave with $2$-step nilpotent Ricci operator.
\end{enumerate} 
\end{theorem}
\begin{proof}
We analyze both cases separately. Assume first that $g(\nh,\nh)\neq 0$. Note that, for any non-isotropic solution, since $\nh$ is not lightlike, we can consider a local pseudo-orthonormal frame $\mathcal{B}=\{E_1,E_2,\dots,E_n\}$, where $E_1=\frac{\nh}{|\nh|}$ and $|\nh|=\sqrt{\varepsilon g(\nh,\nh)}$ ($\varepsilon=\pm1$ depending on whether $\nh$ is spacelike or timelike, respectively). Furthermore, without loss of generality, we can take $g(E_2,E_2)=\varepsilon_2=-\varepsilon$ and $g(E_i,E_i)=1$ for $i>2$.
 
On the one hand, in the expression \eqref{eq:tensorD2} for the augmented Cotton tensor $D$, take $Y=\frac{E_1}{|\nh|}$ so that $g(\nh,Y)=\varepsilon$, and take $X=E_i$, $Z=E_j$, $i,j>1$. Since $D$ vanishes, we have $\rho(E_i,E_j)=\frac{\tau-\varepsilon\rho(E_1,E_1)}{n-1}g(E_i,E_j)$.
	Then, by equation \eqref{eq:vacuum-Einstein-field-equations-2},
	\[
		\Hes_h(E_i,E_j)=-h\varepsilon\frac{\rho(E_1,E_1)}{n-1} g(E_i,E_j).
	\]
	On the other hand, we can take $X=Y=E_1$ and $Z=E_i$, $i>1$ to find $\rho(E_1,E_i)=\Hes_h(E_1,E_i)=0$. It follows that the level hypersurfaces of $h$ in $M$ are totally umbilical and, furthermore, that the distribution generated by $\nh$ is totally geodesic. Consequently, $(M,g)$ splits locally as a twisted product $I\times_\varphi N$, where $I\subset \mathbb{R}$ is an open interval, for some function $\varphi$ on $I\times N$. Moreover, the mean curvature vector field $\nabla h$ is parallel in the normal bundle ($\nabla^\perp \nh=0$, where $\nabla^\perp$ is the normal connection). Indeed, for $i\neq1$, $\varepsilon\nabla^\perp_{E_i}\nh=g(\nabla_{E_i}\nh,E_1)=\Hes_h(E_1,E_i)=0$. Hence the leaves of the fiber form a spherical foliation and the twisted product reduces to a warped product $I\times_\varphi N$ for some function $\varphi$ on $I$ (see \cite{Ponge-Twisted}).
	
	
Let $t$ be a local coordinate parameterizing $I$ by arc length with $E_1=\nabla t=\varepsilon \partial_t$ and let $\varepsilon\alpha=\rho(E_1,E_1)$ and $\lambda=\frac{\tau- \alpha}{n-1}$. Then, we can write $\Hes_{h}(E_1,E_1)=h''$ and, by the weighted Einstein field equations \eqref{eq:vacuum-Einstein-field-equations-2}, $\alpha=\varepsilon h^{-1}h''+2J$, so $\alpha$ depends only on $t$. Moreover, since $\tau$ is constant, $\rho(E_i,E_j)=\lambda g(E_i,E_j)$ depends only on $t$ as well. We shall show that $N$ is Einstein as follows. Consider the basis $\{\bar{E_i}=\varphi E_i\}_{i=2,\dots,n}$ which is orthonormal on $N$. From the expression of the Ricci tensor of a warped product (see \cite{Oneill}), we have
\begin{equation}\label{eq:fiber1Einstein}
\begin{array}{rcl}
 	\rho^N(\bar{E_i},\bar{E_j})&=&\rho(\bar{E_i},\bar{E_j})+\varepsilon g(\bar{E_i},\bar{E_j})\left(\frac{\varphi''}{\varphi}+(n-2)\frac{(\varphi')^2}{\varphi^2}\right) \\
	\noalign{\medskip}
	&=& \varphi^2 \left(\varepsilon\lambda+\frac{\varphi''}{\varphi}+(n-2)\frac{(\varphi')^2}{\varphi^2}\right)\varepsilon g(E_i,E_j).
\end{array}
\end{equation}
	Since $\rho^N(\bar{E_i},\bar{E_j})$ is a function defined on the fiber, it does not depend on $t$, which is a coordinate of the base. Hence,  $\rho^N=\beta g^N$  for some $\beta\in \mathbb{R}$ and $N$ is Einstein. Thus, assertion (1) holds.

Now, assume $g(\nabla h,\nabla h)= 0$. Then, it follows from \cite{Brozos-Mojon} that the Ricci operator is nilpotent. Moreover, there exists a pseudo-orthonormal basis $\{\nabla h, U, X_1, \dots, X_{n-2}\}$ such that the non-zero terms of the metric tensor are $g(\nabla h,U)=g(X_i,X_i)=1$, $i=1,\dots, n-2$, and the Ricci operator is given by $\operatorname{Ric}(U)=\nu\nabla h+\mu X_1$ and $\operatorname{Ric}(X_1)=\mu \nabla h$ (see \cite{Brozos-Mojon,Oneill}). 

Since $D=0$, equation \eqref{eq:tensorD2} evaluated on $(U,U,X_1)$ yields
\[
\begin{array}{rcl}
	0&=&(n-1) \rho\wedge dh(U,U,X_1)+ g\wedge \iota_{\nh}\rho(U,U,X_1)-\tau  g\wedge dh(U,U,X_1)\\
	\noalign{\medskip}
	&=& - (n-1) \mu .
\end{array}
\]
Hence, $\mu=0$, so the Ricci operator is two-step nilpotent and it follows from \cite{Brozos-Mojon} that $(\mathfrak{U},g|_\mathfrak{U})$ is a Brinkmann wave.
\end{proof}

Note that the warped product structure of non-isotropic solutions with vanishing $D$ tensor, described in Theorem~\ref{th:loc-conf-flat-nisolutions}~(1), is analogous to the case of Riemannian signature discussed in \cite{Qing-Yuan}. This analogy also works when considering locally conformally flat non-isotropic solutions and comparing them to those studied in Riemannian signature in \cite{Kobayashi} as we will see in the following subsection.

\subsection{Locally conformally flat solutions}
We will start this subsection keeping the dimension of the manifold arbitrary in order to prove Theorem~\ref{th:loc-conf-flat-ndim}, and then we will obtain specific results in dimension four. Unsurprisingly, the vanishing of the Weyl tensor turns out to be more restrictive than the vanishing of the augmented Cotton tensor.

\vspace{1em}

\noindent {\it Proof of Theorem~\ref{th:loc-conf-flat-ndim}}:
For locally conformally flat solutions, the augmented Cotton tensor $D$ given by \eqref{eq:tensorD1} vanishes identically since $W=0$ implies $dP=0$, so we apply Theorem~\ref{th:loc-conf-flat-nisolutions} to obtain a local splitting into a warped product $I\times_\varphi N$ if $g(\nh,\nh)\neq 0$ and adopt the notation of its proof. A warped product of the form $I\times_\varphi N$ is locally conformally flat if and only if the fiber $N$ has constant sectional curvature (see \cite{Brozos-Loc-Conf-Flat}).

Once the local splitting into a warped product has been established, we use the expressions for the Ricci tensor of a warped product (see \cite{Oneill}) and the weighted Einstein field equations \eqref{eq:vacuum-Einstein-field-equations-2} to compute the Laplacian of $h$, the two eigenvalues of the Ricci operator ($\operatorname{Ric}(\nh)=\alpha \nh$ and $\operatorname{Ric}(E_i)=\lambda E_i$ for $i>1$) and the scalar curvature in terms of $h$ and $\varphi$ to obtain:
\[
\begin{array}{rclrll}
	\Delta h&=&\varepsilon\left(h''+(n-1)\frac{h'\varphi'}{\varphi}\right), \quad &\alpha&=-\varepsilon(n-1)\frac{\varphi''}{\varphi}=\varepsilon \frac{h''}h+\frac{\tau}{n-1},  \\
	\noalign{\medskip}
	\lambda&=&\frac{\tau- \alpha}{n-1}= \varepsilon \frac{\varphi''}{\varphi}+\frac{\tau}{n-1}, \quad &\tau&=-(n-1)\varepsilon(\frac{h''}h+(n-1)\frac{h'\varphi'}{h\varphi}).
\end{array}
\]
Note that, for $\tau$, we have used the second expression for $\alpha$ and the fact that, by \eqref{eq:vacuum-Einstein-field-equations-2}, $h\left(\lambda-\frac{\tau}{n-1}\right)g(E_i,E_i)=\Hes_h(E_i,E_i)= \varepsilon\frac{\varphi' h'}{\varphi}g(E_i,E_i)$, where we have used \eqref{eq:connection-warped-product} to obtain the last expression.

The non-diagonal terms of $G^h=h\rho-\Hes_h+\Delta h g$ vanish identically. Hence we compute the diagonal terms as follows:
\[
\begin{array}{rcl}
	0=G^h(E_1,E_1)&=&-(n-1)\frac{h\varphi''}{\varphi}+h''+(n-1)\frac{h'\varphi'}{\varphi}-h''	 \\
	\noalign{\medskip}
	&=&(n-1)\left(\frac{h'\varphi'-h\varphi''}{\varphi}\right), \\
	\noalign{\medskip}
	0=G^h(E_i,E_i)&=&\varepsilon\left(h\left( \frac{\varphi''}{\varphi}+\varepsilon\frac{\tau}{n-1}\right)+\left(h''+(n-1)\frac{h'\varphi'}{\varphi}\right)-\frac{h'\varphi'}{\varphi}\right)g(E_i,E_i)	 \\
	\noalign{\medskip}
	&=&\varepsilon\left(h''+h(n-1) \frac{\varphi''}{\varphi}+h \varepsilon \frac{\tau}{n-1}\right)g(E_i,E_i),
\end{array}
\]
where we have used the relation $h'\varphi'-h\varphi''=0$ that we got from the first expression to simplify the second one.
Hence we obtain that the system of ODEs given in \eqref{eq:weighted-field-ODEs} are necessary and sufficient conditions for a warped product as above to be a solution of \eqref{eq:vacuum-Einstein-field-equations}.

Assume now that $g(\nh,\nh)=0$ on an open subset $\mathfrak{U}\subset M$. Since $D=0$, we use Theorem~\ref{th:loc-conf-flat-nisolutions} to see that the Ricci tensor is either flat or $2$-step nilpotent. Moreover, $\tau=J=0$ and the only non-zero term of the Ricci tensor is $\rho(U,U)=\nu$.
Therefore, the manifold is a $pp$-wave if and only if $R(\mathcal{D}^\perp,\mathcal{D}^\perp,\cdot,\cdot)=0$ (see \cite{leistner}).  
  
Notice that $\mathcal{D}^\perp=\operatorname{span}\{\nabla h, X_1,\dots,X_{n-2}\}$ and,
since $W=0$, we have that $R=P\KN g=\frac{1}{ n-2}\rho\KN g$. Hence by directly substituting in this expression we get that $R(\mathcal{D}^\perp,\mathcal{D}^\perp,\cdot,\cdot)=0$ and that  $(M,g)$ is indeed a $pp$-wave.

Now, locally conformally flat $pp$-waves are plane waves that admit local coordinates $(u,v,x_1\dots,x_{n-2})$ such that the metric takes de form
\[
g(u,v, x_1,\dots, x_{n-2})= 2 dv du+F(v,x_1,\dots, x_{n-2}) dv^2+\sum_{i=1}^{n-2} dx_i^2,
\]
where $F(v,x_1,\dots, x_{n-2})=\frac{a(v)}{ n-2} \sum_{i=1}^{n-2} x_i^2+\sum_{i=1}^{n-2} b_i(v) x_i+c(v)$ (see, for example, \cite{BV-GR-GF}). With respect to these coordinates, the only non-vanishing component of the Ricci tensor is $\rho(\partial_u,\partial_u)=-a(u)$. Moreover, because the distinguished parallel lightlike distribution of the $pp$-wave corresponds with $\nabla h$ by construction, it follows that $\nabla h$ is a multiple of $\partial_u$, so $h(v,u,x_1,\dots,x_{n-2})=h(v)$. Now, a direct computation of the terms in \eqref{eq:vacuum-Einstein-field-equations} yields the only condition:
\[
-a(v) h(v)-h''(v)=0,
\]
from where case (2) follows. \qed

\begin{remark}\label{re:kobayashi}
The system of ODEs \eqref{eq:weighted-field-ODEs} was obtained in \cite{Kobayashi} for $\varepsilon=1$ in the Riemannian setting. An analogous reasoning to that in \cite{Kobayashi} shows that from \eqref{eq:weighted-field-ODEs}, it follows that
\[
\begin{array}{rcl}
	\gamma&=&\varphi^{n-1}\varphi''+\frac{\varepsilon \tau}{n(n-1)}\varphi^n, \\
	\noalign{\medskip}
	\frac{\varepsilon\kappa}{n-2}&=&(\varphi')^2+\frac{2\gamma}{n-2}\varphi^{2-n}+\frac{\varepsilon \tau}{n(n-1)}\varphi^2,
\end{array}
\]
where $\gamma,\kappa$ are real constants and $\rho^N=\kappa g^N$. The discussion of the solutions to these ODEs in \cite{Kobayashi} in terms of the constants $\tau$, $\kappa$ and $\gamma$, also applies to the Lorenztian case by substituting $\tau$ and $\kappa$ with $\varepsilon \tau$ and $\varepsilon \kappa$ respectively. Note that these ODEs are also satisfied in the case $D=0$ (not necessarily locally conformally flat), if we allow a generic Einstein fiber instead of a fiber of constant sectional curvature.
\end{remark}

Isotropic locally conformally flat solutions are completely characterized, for arbitrary dimension, in Theorem~\ref{th:loc-conf-flat-ndim}~(2). For the non-isotropic case, we give a detailed description of solutions in dimension four as follows, where the nature of solutions depends on the sign of the scalar curvature, which is constant by Lemma~\ref{le:const-sc}.

\begin{corollary}\label{th:loc-conf-flat-warped}
	Let $(M,g,h)$ be a non-isotropic, non-flat solution of the vacuum weighted Einstein field equations with $\dim M=4$ and vanishing augmented Cotton tensor. Then $M$ decomposes locally as a product $I\times N$, where $I\subset \mathbb{R}$ is an open interval with $\nh$ tangent to $I$; and $N$ is a $3$-dimensional manifold with constant sectional curvature $\kappa$. Moreover, the metric and the density functions satisfy one of the following:
		\begin{enumerate}
			\item $g$ is a direct product metric $\varepsilon dt^2+g^N$ with $t$ a coordinate  parameterizing $I$ by arc length such that 
			\[
			\begin{array}{rclr}
				h(t)&=&c_1\sin\left(\frac{\sqrt{2\varepsilon\kappa}}\varphi t\right)+c_2\cos\left(\frac{\sqrt{2\varepsilon\kappa}}\varphi t\right), \,&\text{ if } \varepsilon\kappa>0, \\
				\noalign{\medskip}
				h(t)&=&c_1e^{\frac{\sqrt{-2\varepsilon\kappa}}{\varphi} t}+c_2e^{- \frac{\sqrt{-2\varepsilon\kappa}}{\varphi} t}, \,&\text{ if } \varepsilon\kappa<0,
			\end{array}
			\]
			\item $g$ is a warped product metric $\varepsilon dt^2+\varphi(t)^2g^N$ with $t$ a coordinate  parameterizing $I$ by arc length such that the density function $h$ satisfies $h(t)=A\varphi'(t)$, $A\in \mathbb{R}^*$, and $\varphi$ takes the following forms, depending on the sign of the scalar curvature $\tau$ of the product:
			\[
			\begin{array}{rclrl}
				\varphi(t)^2&=&\frac{6\kappa}{\tau}+c_1\sin\left(\sqrt{\frac{\varepsilon\tau}{3}}t\right)+c_2\cos\left(\sqrt{\frac{\varepsilon\tau}{3}}t\right), \,&&\text{ if } \varepsilon\tau>0, \\
				\noalign{\medskip}
				\varphi(t)^2&=&\frac{6\kappa}{\tau}+c_1e^{\sqrt{-\frac{\varepsilon\tau}{3}}t}+c_1e^{-\sqrt{-\frac{\varepsilon\tau}{3}}t}, \,&&\text{ if } \varepsilon\tau<0, \\
				\noalign{\medskip}
				\varphi(t)^2&=& \varepsilon\kappa t^2+c_1t+c_2, \,&&\text{ if } \tau=0,\, c_1^2\neq 4\varepsilon c_2\kappa,
			\end{array}
			\]
		\end{enumerate}
		where $A,c_1,c_2$ are suitable integration constants so that $\varphi^2(t),h(t)>0$ for all $t\in I$. 		
\end{corollary}
\begin{proof}
From the first ODE in \eqref{eq:weighted-field-ODEs}, it follows that either $\varphi'=0$, so we have a Riemannian product, or $h(t)=A\varphi'(t)$ with $A\in \mathbb{R}^*$ such that $h>0$ for all $t\in I$. In the first case, the remaining non-vanishing components of the weighted Einstein field equations take the form
\[
	0=G^h(E_i,E_j)=\varepsilon\left(h''+\varepsilon \frac{\tau}3 h \right)g(E_i,E_j),
\]
where  $\tau=\frac{6\kappa}{\varphi^2}$. We can solve the resulting ODE  $0=h''+\frac{2\varepsilon\kappa}{\varphi^2} h$ to get the density function. Note that, if $\kappa=0$, then the manifold is flat. 

Now, assume that $\varphi'\neq 0$, so $h(t)=A\varphi'(t)$, and take $F(t)=\varphi(t)^2$. Then, we compute the scalar curvature of the warped product, in terms of $\kappa$ and $F$, using the warped product curvature expressions, resulting in the equation  $0= \tau F-3(2\kappa-\varepsilon F'')$. We solve this ODE to get the different forms of $\varphi^2$. Then, a direct computation shows that all components of the weighted Einstein field equations  vanish. Note that, if $\tau=0$ and $c_1^2=  4\varepsilon c_2\kappa$, the manifold is flat.
\end{proof}

For manifolds of dimensions other than 4, it is not possible to express all possible solutions in such a simple way as in Corollary~\ref{th:loc-conf-flat-warped}. However, we refer to Section~\ref{sec:4-dim-diagonal} for some generalizations of Kobayashi's locally conformally flat static spaces to Lorentzian signature that are of special interest, since they appear as submanifolds of higher-dimensional solutions. The following example will be used to illustrate this fact.

\begin{example}\label{ex:3-dim-loc-conf-flat-sols}
	Let $(I\times_\varphi N,g,h)$ be a 3-dimensional (Riemannian or Lorentzian) SMMS, with $N$ a surface of constant Gauss curvature $\kappa$. If this triple is a non-flat solution of the weighted Einstein field equations  \eqref{eq:vacuum-Einstein-field-equations} with vanishing scalar curvature $\tau$, then the system of ODEs \eqref{eq:weighted-field-ODEs} is satisfied and, moreover, $\gamma=\varphi^2 \varphi''$ and $ \varepsilon\kappa=(\varphi')^2+2\gamma\varphi^{-1}$ for some constant $\gamma\in \mathbb{R}\backslash\{0\}$, since the manifold becomes flat if $\gamma=0$. If $\gamma>0$, this corresponds to case IV.I in \cite{Kobayashi}, and if $\gamma<0$, to case III.1 (substituting $\kappa$ by $\varepsilon \kappa$). Notice that all solutions of this kind have two distinct Ricci eigenvalues.
\end{example}


\section{Solutions with harmonic curvature. The diagonalizable case}\label{sec:4-dim-diagonal}

We have already seen how local conformal flatness only allows for very specific warped product structures (see Theorem~\ref{th:loc-conf-flat-ndim} and Corollary~\ref{th:loc-conf-flat-warped}) and how they relate to the Riemannian static spaces discussed in \cite{Kobayashi}. In order to get a broader family of solutions with a more flexible geometry, we are going to focus on dimension four and impose a weaker restriction than local conformal flatness: harmonic Weyl tensor. Since every solution has constant scalar curvature (see Lemma~\ref{le:const-sc}), this is equivalent to the curvature tensor being harmonic. 

Our analysis is divided into several sections depending on the structure of the Ricci operator $\Ric$. We will prove shortly that, for solutions with harmonic curvature, $\nh$ is an eigenvalue of $\Ric$ (see Lemma~\ref{lemma:nh-Ricci-eigen}). If $g(\nh,\nh)\neq 0$, since $\Ric$ is self-adjoint, at each point of the manifold, it takes one of the following four forms (see \cite{Oneill}): On the one hand, relative to an orthonormal frame $\mathcal{B}_1=\{E_1=\nh/|\nh|,E_2,E_3,E_4\}$,
\begin{equation}\label{eq:matrices-type-IaandIb}
\Ric=\begin{pmatrix}
\lambda_1 & 0 & 0 & 0 \\
0 & \lambda_2 & 0 & 0 \\
0 & 0 & \lambda_3 & 0 \\
0 & 0 &  0& \lambda_4 
\end{pmatrix} \quad \text{or} \quad 	
	\Ric=\begin{pmatrix}
\lambda & 0 & 0 & 0 \\
0 & a & b & 0 \\
0 & -b & a & 0 \\
0 & 0 &  0& \alpha 
\end{pmatrix},
\end{equation}
with $b\neq 0$. Following standard terminology, we refer to these structures as \textit{Type I.a} and \textit{Type I.b}, respectively.

On the other hand, relative to a pseudo-orthonormal frame of the form $\mathcal{B}_2=\{ E_1=\nh/|\nh|,U,V,E_2\}$, where the only non-vanishing terms of the metric are $g(E_i,E_i)=1$, $i=1,2$, $g(U,V)=1$, there are two more possible forms:
\begin{equation}\label{eq:matrices-type-IIandIII}
\Ric=\begin{pmatrix}
\lambda & 0 & 0 & 0 \\
0 & \alpha & 0 & 0 \\
0 & \varepsilon & \alpha & 0 \\
0 & 0 &  0& \beta 
\end{pmatrix} \quad \text{or} \quad 	\Ric=\begin{pmatrix}
\lambda & 0 & 0 & 0 \\
0 & \alpha & 0 & 1 \\
0 & 0 & \alpha & 0 \\
0 & 0 & 1 & \alpha 
\end{pmatrix}  ,
\end{equation}
which we call \textit{Type II} and \textit{Type III} respectively. 

For solutions with $g(\nh,\nh)=0$, the Ricci operator is nilpotent \cite{Brozos-Mojon}. Moreover, there exists an adapted pseudo-orthonormal frame $\mathcal{B}_0=\{\nabla h, U, X_1, X_2\}$ such that the non-zero terms of the metric tensor are $g(\nabla h,U)=g(X_i,X_i)=1$, $i=1,2$, and the Ricci operator takes the form
\begin{equation}\label{eq:matrices-iso}
\Ric=\begin{pmatrix}
0 & \nu & \mu & 0 \\
0 & 0 & 0 & 0 \\
0 & \mu & 0 & 0 \\
0 & 0 & 0 & 0 
\end{pmatrix}  .
\end{equation}
Hence, the isotropic solution is Ricci-flat (corresponding to Type I.a), 2-step nilpotent ($\mu=0$ and $\nu \neq 0$, Type II) or 3-step nilpotent ($\mu\neq 0$, Type III).

Henceforth, we will assume that the Ricci operator of any solution is of constant type in the manifold. Otherwise, one would restrict to an open subset where this happens. In this section, we treat the diagonalizable case (see Section~\ref{sec:non-diagonalizable-complex} for a study of Type I.b and Section~\ref{sec:non-diagonalizable-real} for details on Type II and Type III).  Solutions with harmonic curvature were previously considered in Riemannian signature in \cite{Kim-Shin}, where the Ricci operator is necessarily diagonalizable. Motivated by this work, we follow some of the arguments applied to static spaces to obtain all possible solutions in this setting (see also \cite{Derdzinski-Codazzi} for the study of eigendistributions of the Ricci operator on manifolds with harmonic curvature, whose arguments will be mimic at some instances, and \cite{cao-chen} for related arguments for Ricci solitons).
Unsurprisingly, some of the results in this section are reminiscent of those in \cite{Kim-Shin}. However, the fact that we are working in Lorentzian signature allows for greater flexibility and gives rise to new geometric structures when  the solution is isotropic. Moreover, if $\nh$ is timelike, we will see that the Ricci operator is necessarily diagonalizable, so all solutions with this character of $\nh$ are described below in Theorem~\ref{th:classification-diagonalizable}.

Much like in the Riemannian case, the geometric structure of a non-isotropic solution strongly depends on the number of distinct eigenvalues of $\Ric$. Arguments in \cite{Derdzinski-Codazzi} and \cite{Kim-Shin} show that this number does not change in an open dense subset of $M$. Indeed, for $x\in M$, let $E_{\Ric}(x)$ be the number of distinct eigenvalues of $\Ric_x$ and set $M_{\Ric}=\{x\in M \colon E_{\Ric} \text{ is constant in a neighborhood of } x\}$. It is clear that $M_{\Ric}$ is open.
	To show that this subset is dense, take $x\in M$ and consider any open ball $B$ centered at $x$. Since the rank of $E_{\Ric}$ is finite, there is a point $q\in B$ where $E_{\Ric}(q)$ is the maximum of $E_{\Ric}$ on $B$. Since a small variation of the eigenvalues cannot decrease the value $E_{\Ric}(q)$ and because it is maximum by definition, there is a neighborhood of $q$ where  $E_{\Ric}= E_{\Ric} (q)$, so $q\in M_{\Ric}$.
	Therefore, since we are working locally, we will treat those points that belong to  $M_{\Ric}$ for the non-isotropic case.

\begin{theorem}\label{th:classification-diagonalizable}
Let $(M,g,h)$ be a $4$-dimensional smooth metric measure space with diagonalizable Ricci operator and harmonic curvature (not locally conformally flat).

\begin{enumerate}
	\item  If $g(\nh,\nh)=0$, then $(M,g)$ is a solution of equation~\eqref{eq:vacuum-Einstein-field-equations} if and only if $(M,g)$  is a Ricci-flat $pp$-wave and, in appropriate local coordinates $\{u,v,x_1,x_2\}$, it can be written as
	\begin{equation}\label{eq:local-coordinates-pp-wave}
	g(u,v,x_1,x_2)=2 du dv+ F(v,x_1,x_2) dv^2+dx_1^2+dx_2^2,
	\end{equation}
	with $\Delta_x F=\partial_{x_1}^2 F+\partial_{x_2}^2 F=0$, and $h(u,v,x_1,x_2)=v$.
	\item  If $g(\nh,\nh)\neq 0$ and $(M,g,h)$ is a solution to \eqref{eq:vacuum-Einstein-field-equations}, then  $\left(M_{\Ric},\restr{g}{M_{\Ric}}\right)$ is locally isometric to:
	\begin{enumerate}
		\item   A direct product $I_2\times \tilde M$, where $\tilde M=I_1\times_\xi N$ is a warped product $3$-dimensional solution with $\tilde \tau=0$ and $N$ a surface of constant Gauss curvature. Moreover, $h=c\xi'$ is defined on $I_1$. 
		\item  A direct product $N_1\times N_2$ of two surfaces of constant Gauss curvature $\frac{\kappa}{2}$ and $\kappa$, respectively. The density function is defined on $N_1$ and is a solution to the Obata equation  $\Hes_{h}^{N_1}=-\frac{\kappa h}2 g^{N_1}$.
	\end{enumerate}
\end{enumerate}
\end{theorem}

\begin{remark}
	Notice that the condition on the defining function $F$ of the $pp$-wave metric in \eqref{eq:local-coordinates-pp-wave} resembles the Laplace equation. Thus, consider, for example, a solution of the form $F(v,x_1,x_2)=f(v) (x_1^2-x_2^2)$ to build solutions with harmonic Weyl tensor but which are not locally conformally flat. Indeed, the non-vanishing components of the Weyl tensor are, up to symmetries, \[W(\partial_v,\partial_{x_1},\partial_v,\partial_{x_1})=W(\partial_v,\partial_{x_2},\partial_v,\partial_{x_2})=-f(u).\]
\end{remark}

\begin{remark}
	Solutions in Theorem~\ref{th:classification-diagonalizable}~(2)-(a) are build from $3$-dimensional locally conformally flat solutions with vanishing scalar curvature (see Example~\ref{ex:3-dim-loc-conf-flat-sols}), just by adding an Euclidean factor, and result in a multiply warped product of the form $I_1\times I_2\times_\xi N$. 	
	Note that the timelike direction corresponds either to the factor $I_1$ or the factor $I_2$. If $I_1\times_\xi N$ is Riemannian with $N$ of constant positive Gauss curvature ($\kappa>0$), it was pointed out by Kobayashi that there is a solution in $\mathbb{R}\times S^2$ which contains a spatial slice of the well-known Schwarzschild spacetime. Thus, we  construct a solution on the 4-dimensional spacetime $\mathbb{R}\times (\mathbb{R}\times_\varphi S^2)$, with the metric given by $g=-ds^2\oplus g^{\textrm{Sch}}$, where  $g^{\textrm{Sch}}$ stands for the spatial part of the Schwarzschild metric (see \cite{stephani-et-al} for details on this solution). In contrast, solutions for $\kappa<0$ are incomplete  (cf. \cite[Example 3]{Kim-Shin}). 
		
From Example~\ref{ex:3-dim-loc-conf-flat-sols}, taking $\kappa=0$ allows for an explicit expression  for $\varphi(t)$. Indeed, if $\kappa=0$, then $(\varphi')^2+2a\varphi^{-1}=0$ and $a=\varphi^2\varphi''$. Thus, we can write $(\varphi')^2+2\varphi\varphi''=0$. Solving this ODE yields (after a translation of $t$, if needed) $\varphi(t)=K_1t^{2/3}$ and $h(t)=K_2 t^{-1/3}$ for some suitable $K_1,K_2\in \mathbb{R}^*$. This gives solutions with Ricci eigenvalues $\left\{ \frac{4\varepsilon}{9t^2},0\right\}$ (simple) and $\left\{ -\frac{2\varepsilon}{9t^2}\right\}$ (double).
\end{remark}

\noindent{\it Proof of Theorem~\ref{th:classification-diagonalizable}~(1).}
If $g(\nh,\nh)=0$, it was shown in \cite{Brozos-Mojon} that the manifold is necessarily Ricci-flat (hence all solutions of this type have harmonic curvature tensor), and $\Hes_h=0$, so the solution is a Ricci-flat Brinkmann wave with parallel vector field $\nh$. In this case, because $\nh$ is parallel, $R(\nh,X,Y,Z)=0$ for all vector fields $X$, $Y$, $Z$. Moreover, on a pseudo-orthonormal frame $\{\nh,U,X_1,X_2\}$, we have
\[
\begin{array}{rcl}
	0&=& \rho(X_1,X_1)=2 R(X_1,U,X_1,\nh)+R(X_1,X_2,X_1,X_2)=R(X_1,X_2,X_1,X_2),\\
	\noalign{\medskip}
	0&=& \rho(X_1,U)= R(X_1,U,U,\nh)+R(X_1,X_2,U,X_2)=R(X_1,X_2,U,X_2).
\end{array}
\]
Analogously, $R(X_1,X_2,U,X_1)=0$, so we conclude $R(X_1,X_2,\cdot,\cdot)=0$. Consequently,  $R(\nh^\perp,\nh^\perp,\cdot,\cdot)=0$ and the Brinkmann wave is indeed a $pp$-wave (see \cite{leistner}). 
There exist local coordinates $\{u,v,x_1,x_2\}$ so that the metric is given by \eqref{eq:local-coordinates-pp-wave}. A direct computation shows that, the $pp$-wave is Ricci-flat if and only if the spacelike Laplacian vanishes: $\partial_{x_1}^2 F+\partial_{x_2}^2 F=0$. Moreover, $h$ is only a function of $v$ and, since $\Hes_h=0$, we have $h''(v)=0$. Now, the coordinate $v$ can be normalized so that $h(u,v,x_1,x_2)=v$.
\qed

\vspace{1em}

Non-isotropic solutions require a deeper analysis to provide the classification in Theorem~\ref{th:classification-diagonalizable}~(2). 
Thus, throughout the rest of this section, all solutions are assumed to be non-isotropic. Following ideas developed in \cite{Kim-Shin} for the Riemannian counterpart, we establish some preliminary results.  Although we are focusing on 4-dimensional manifolds, they apply to solutions of arbitrary dimension.

\begin{lemma}\label{lemma:nh-Ricci-eigen}
	For any $n$-dimensional solution $(M,g,h)$ of the weighted Einstein field equations \eqref{eq:vacuum-Einstein-field-equations} with harmonic curvature, $\Ric(\nh)=\lambda \nh$ for some smooth function $\lambda$ on $M$.
\end{lemma}
\begin{proof}
 Assume $\operatorname{div} R=dP=0$. In \eqref{eq:Rnf}, we can choose $X=\nh$, $Y\perp \nh$ and $Z$ such that $g(Z,\nh)=1$ to see that
\[
\begin{array}{rcl}
	0&=&R(\nh,\nh,Y,Z) \\
	\noalign{\smallskip}
	&=&dh(Z)(\rho-2Jg)(Y,\nh)-dh(Y)(\rho-2Jg)(Z,\nh)=\rho(Y,\nh)
\end{array}
\]
for every $Y\perp \nh$. Consequently, $\nh$ is an eigenvector of the Ricci operator $\Ric$.
\end{proof}

The fact that $\nh$ is a real eigenvector of $\Ric$ has important geometric consequences for vacuum solutions. For example, if $\nh$ is timelike, since $\Ric$ is self-adjoint, it follows that $\Ric$ diagonalizes in a suitable pseudo-orthonormal frame $\mathcal{B}_1=\{E_1,E_2,\dots,E_4\}$, where $E_1=\nh/|\nh|$ and $|\nh|=\sqrt{\varepsilon g(\nh,\nh)}$ (with $\varepsilon=g(E_1,E_1)=-1$), $g(E_2,E_2)=-\varepsilon$ and $g(E_i,E_i)=1$ for $i>2$. (see \cite{Oneill}). We will refer to this as an {\it adapted frame}. Consequently, all solutions with $\nh$ timelike are described by Theorem~\ref{th:classification-diagonalizable}.
 
 Note that, if $\nh$ is spacelike, this structure is also possible, with $\varepsilon=1$, but other Jordan forms for $\Ric$ can also arise.
Since we are assuming that $\Ric$ diagonalizes, let $\Ric E_i=\lambda_i E_i$. From the vacuum equation \eqref{eq:vacuum-Einstein-field-equations}, it follows that the Hessian operator $\hes_h=\nabla \nh$ diagonalizes in the frame  $\mathcal{B}_1$, with $\hes_hE_i=h(\lambda_i-2J)E_i$. In particular, this implies that $X(g(\nh,\nh))=2\Hes_h(\nh,X)=0$ for all $X\perp \nh$, so the distribution generated by $\nh$ is totally geodesic. This also means that $|\nh|$ is constant on each connected component of the level sets of $h$, so we can write $\nabla_{E_i}E_1=\beta_iE_i$ with $\beta_i=\frac{h(\lambda_i-2J)}{|\nh|}$.  Furthermore, we have 
\[
	0=E_i(g(\nh,E_j))=g(\nabla_{E_i}\nh,E_j)+g(\nh,\nabla_{E_i}E_j)=g(\nh,\nabla_{E_i}E_j)
\]
for $i,j> 1$, $i\neq j$. It follows that $\mathrm{span}\{E_2,\dots, E_4\}$ is closed under Lie bracket, and the distribution generated by $\mathrm{span}\{E_2,\dots, E_n\}$ is integrable, so $M$ splits locally as a product $I\times N$, where $I$ is an open interval to which $E_1$ is tangent, and $N$ projects onto the leaves of the foliation generated by $\mathrm{span}\{E_2,\dots, E_n\}$. Moreover, we have
\[
	 d\left(\frac{dh}{|\nh|}\right)=-\frac{1}{2|\nh|^3}d|\nh|^2\wedge d h=0,
\]
since $\nabla_X(|\nh|^2)=0$ for $X\perp \nh$. Thus, $dh/|\nh|$ is a closed form, and there is a local coordinate $t$ such that $dt=dh/|\nh|$. Note that $\nabla t=\nh/|\nh|=E_1$, so that $\nabla_{E_1}E_1=0$, and $h=h(t)$. With this, the metric takes the form
\begin{equation}\label{eq:decomposition}
	g^M= \varepsilon dt^2\oplus g^N,
\end{equation}
with $N$ Lorentzian if $\varepsilon=1$ and Riemannian if $\varepsilon=-1$. We compute  $\Hes_{h}(E_1,E_1)=h''$ and, from \eqref{eq:vacuum-Einstein-field-equations-2}, we have $\lambda_1= \varepsilon h^{-1}h''+2J$, so $\lambda_1$ depends only on $t$. With this, we prove the following lemma.

\begin{lemma}\label{lemma:coordinate-s-eigenvalues}
	Let $(M,g,h)$ be an $n$-dimensional non-isotropic solution of the weighted Einstein field equations \eqref{eq:vacuum-Einstein-field-equations} with harmonic curvature such that the Ricci operator $\Ric$ diagonalizes. Then, all eigenvalues of $\Ric$ depend only on the local coordinate $t$.
\end{lemma}
\begin{proof}
From the harmonicity condition $\operatorname{div} R=0$, it follows that the Ricci tensor satisfies $(\nabla_{E_i}\rho)(E_i,\nh)=(\nabla_{\nh}\rho)(E_i,E_i)$. Then, using \eqref{eq:vacuum-Einstein-field-equations-2}, we see that, for $i\neq 1$,
\[
\begin{array}{rcl}
 	(\nabla_{E_i}\rho)(E_i,\nh)&=&E_i(\rho(E_i,\nh))-\rho(\nabla_{E_i}E_i,\nh)-\rho(E_i,\nabla_{E_i}\nh ) \\
	\noalign{\medskip}
	&=& (\lambda_1-\lambda_i)g(E_i,\nabla_{E_i}\nh)  \\
	\noalign{\medskip}
	&=& \varepsilon_i h(\lambda_1-\lambda_i)(\lambda_i-2J),
\end{array}
\] 
where $ \varepsilon_i=g(E_i,E_i)$. On the other hand, $(\nabla_{\nh}\rho)(E_i,E_i)=\varepsilon_i\nh(\lambda_i)$, so $\nh(\lambda_i)=h(\lambda_1-\lambda_i)(\lambda_i-2J)$. Now, since $\tau$ is constant and $\lambda_1$ depends only on $t$, we have that $\tau-\lambda_1=\sum_{i=2}^n\lambda_i$ depends only on $t$. Moreover,
\[
\begin{array}{rcl}
 	0=\nh(\tau)&=&\nabla h(\lambda_1)+\sum_{j=2}^n  h(\lambda_1-\lambda_j)(\lambda_j-2J) \\
	\noalign{\medskip}
	&=&\nabla h(\lambda_1)+h(\lambda_1+2J)\sum_{j=2}^n \lambda_j -h\left(2J(n-1)\lambda_1+\sum_{j=2}^n\lambda_j^{2}\right).
\end{array}
\]
Since $\nabla t=\nh/|\nh|$ and $|\nh|$ depends only on $t$, every term in the equation above, except for $\sum_{j=2}^n\lambda_j^{2}$, depends only on $t$. Thus, $\sum_{j=1}^n\lambda_j^{2}$ depends only on $t$ as well. We can perform this same process for any $k\in\{1,\dots,n-1\}$ by induction. 
\[
\begin{array}{rcl}
 	k^{-1}\nabla h\left(\sum_{i=1}^n\lambda_i^k\right)&=&\lambda_1^{k-1}\nh(\lambda_1)+\sum_{j=2}^n\lambda_j^{k-1}\nh(\lambda_j) \\
	\noalign{\medskip}
	&=&\lambda_1^{k-1}\nh(\lambda_1)+h\sum_{j=2}^n\lambda_j^{k-1}(\lambda_1-\lambda_j)(\lambda_j-2J) \\
	\noalign{\medskip}
	&=&\lambda_1^{k-1}\nh(\lambda_1)+h(\lambda_1+2J)\sum_{j=2}^n\lambda_j^k \\
	\noalign{\medskip}
	&&-h\left(2J\lambda_1\sum_{j=2}^n\lambda_j^{k-1}+\sum_{j=2}^n\lambda_j^{k+1}\right).
\end{array}
\]
By assumption, every term in the equation above, except for $\sum_{j=2}^n\lambda_j^{k+1}$, depends only on $t$. Thus, $\sum_{j=1}^n\lambda_j^{k+1}$ depends only on $t$ as well. As a result, each $\lambda_i$, $i=1,\dots, n$ depends only on $t$.  In particular, we have $E_i(\beta_j)=E_i(\lambda_j)=0$ for all $i,j=2,\dots,n$.
\end{proof}

It is well-known that the curvature tensor is harmonic if and only if the Ricci tensor is Codazzi. Previous results from \cite{Derdzinski-Codazzi} show that this Codazzi character imposes by itself some important restrictions on the geometry of the leaves of the eigendistributions of $\Ric$. Indeed, a version of the following result was proved in \cite{Derdzinski-Codazzi} and extends from Riemannian to Lorentzian signature when $\Ric$ is diagonalizable. We include the proof here in the interest of completeness and because we will use the weighted Einstein equation to provide additional information on the connection relations for a solution (see also \cite{Kim-Shin}).

\begin{lemma}[cf. \cite{Derdzinski-Codazzi}]\label{lemma:eigen-integrable}
	Let $(M,g,h)$ be an $n$-dimensional non-isotropic solution of the weighted Einstein field equations \eqref{eq:vacuum-Einstein-field-equations} with Codazzi and diagonalizable Ricci tensor.
Then, the distribution associated to each eigenvalue of $\Ric$ is integrable and their corresponding leaves are totally umbilical submanifolds of $M$.
\end{lemma}
\begin{proof}
We work in an adapted local pseudo-orthonormal frame $\mathcal{B}_1=\{E_1,\dots,E_n\}$ that diagonalizes the Ricci operator and such that $g(E_i,E_i)=\varepsilon_i$.
Denote by $\Gamma_{ijk}=g(\nabla_{E_i}E_j,E_k)$ the corresponding Christoffel symbols. First, note that, $\Gamma_{ijk}=-\Gamma_{ikj}$ for all $i,j,k$. We calculate the covariant derivative of the Ricci tensor,
\[
\begin{array}{rcl}
	(\nabla_{E_i}\rho)(E_j,E_k)&=&E_i(\rho(E_j,E_k))-\rho(\nabla_{E_i}E_j, E_k)-\rho(E_j,\nabla_{E_i}E_k) \\
	\noalign{\medskip}
	&=&\varepsilon_j\delta_{jk}E_i(\lambda_j)+(\lambda_j-\lambda_k)\Gamma_{ijk}.
\end{array}
\]
Now, since $\rho$ is Codazzi, we have
\begin{equation}\label{eq:rel-eigen}
	\varepsilon_j\delta_{jk}E_i(\lambda_k)+(\lambda_j-\lambda_k)\Gamma_{ijk}=\varepsilon_{i}\delta_{ik}E_j(\lambda_k)+(\lambda_i-\lambda_k)\Gamma_{jik}.
\end{equation}
From here, choosing $i,j>1$, $i\neq j$, we obtain
\[
	(\lambda_i-\lambda_j)\Gamma_{1ij}=(\lambda_1-\lambda_j)g(\nabla_{E_i}E_1,E_j)=(\lambda_1-\lambda_j)\beta_ig(E_i,E_j)=0.
\]
Hence, for every $i,j$ such that $\lambda_i\neq \lambda_j$, $\Gamma_{1ij}=0$. Furthermore, since $\nabla_{E_1}E_1=0$, we have $\Gamma_{1i1}=0$ and it is clear that $\Gamma_{1ii}=0$ since the adapted frame is pseudo-orthonormal. It follows that $\nabla_{E_1}E_i$ stays in the eigenspace associated to the eigenvalue $\lambda_i$, while being orthogonal to $E_i$. Similarly, $\Gamma_{ii1}=-\varepsilon_i\beta_i$ 
 and, by \eqref{eq:rel-eigen}, $(\lambda_j-\lambda_i)\Gamma_{iji}=0$ if $i,j>1$. Thus, $\Gamma_{iji}=0$ if $\lambda_i\neq \lambda_j$ and $\Gamma_{iii}=0$, so the component of $\nabla_{E_i}E_i$ that is perpendicular to $E_1$ also stays in the eigenspace associated with $\lambda_i$, while being orthogonal to $E_i$.

For the rest of the connection coefficients $\Gamma_{ijk}$, with $i,j,k>1$, we use \eqref{eq:rel-eigen} to write $(\lambda_j-\lambda_k)\Gamma_{ijk}=(\lambda_i-\lambda_k)\Gamma_{jik}$. It follows that, if $\lambda_i=\lambda_k\neq \lambda_j$, $-\Gamma_{ikj}=\Gamma_{ijk}=0$. In other words, if $E_i\neq E_k$ belong to the same eigenspace, then $\nabla_{E_i}E_k$ stays in it. In summary, let $E_i,E_j$ be vectors in the same eigenspace (we denote the set of indices corresponding to eigenvectors in the eigenspace associated to $\lambda_i$ by $[i]$), and $E_\mu$ so that $\lambda_\mu\neq \lambda_i$. Then, in general, the connection relations read
\begin{equation}\label{eq:connection-coefs}
\begin{array}{rcl}
 	\nabla_{E_1}E_1&=&0, \quad \nabla_{E_i}E_1=\beta_iE_i, \quad \nabla_{E_1}E_i=\sum_{k\in[i]}\varepsilon_k \Gamma_{1ik} E_k, \\
		\noalign{\medskip}
		\nabla_{E_i}E_j&=&-\varepsilon_1\varepsilon_i\beta_i \delta_{ij}E_1+\sum_{k\in [i]}\varepsilon_k\Gamma_{ijk} E_k,	\\
		\noalign{\medskip}
		\nabla_{E_i}E_\mu&=&\sum_{k\notin [i], k\neq 1,\mu}\varepsilon_k\Gamma_{i\mu k}E_k.
		\end{array}
\end{equation}
In particular, for two vectors in the same eigenspace, $[E_i,E_j]=\sum_{k\in [i]}\varepsilon_k(\Gamma_{ijk}-\Gamma_{jik}) E_k$, and so the distribution generated by all eigenvectors associated to $\lambda_i$ is integrable. Moreover, the second fundamental form satisfies $II(E_i,E_j)=-\varepsilon_1\varepsilon_i\beta_i \delta_{ij}E_1$, so the tangent submanifold is umbilical.
\end{proof}

 
From this point on, we focus on 4-dimensional solutions to attain the classification in Theorem~\ref{th:classification-diagonalizable}~(2). Once we are working around a point in $M_{\Ric}$, we perform specific analyses depending on whether the eigenvalues $\lambda_2$, $\lambda_3$ and $\lambda_4$ are all different; or at least two of them coincide. As it turns out, the first case is not admissible, independently of the causal character of $\nh$.

\subsection{The three eigenvalues coincide: $\lambda_2= \lambda_3= \lambda_4$}
If $\lambda_2= \lambda_3= \lambda_4$ then, from \eqref{eq:connection-coefs}, the connection behaves as follows
\[
	\begin{array}{rcl}
		\nabla_{E_1}E_1&=&0, \quad \nabla_{E_i}E_1=\beta_iE_i, \quad \nabla_{E_1}E_i=\sum_{k\neq i}\varepsilon_k\Gamma_{1ik} E_k \\
		\noalign{\medskip}
		\nabla_{E_i}E_j&=&-\varepsilon_1\varepsilon_i\beta_i \delta_{ij}E_1+\sum_{k\neq i}\varepsilon_k\Gamma_{ijk} E_k,
	\end{array}
\]
where $i,j,k\in \{2,3,4\}$, $i\neq k$. We consider the distribution $\operatorname{span}\{E_2,E_3,E_4\}$, which is integrable and whose tangent leaves are umbilical (see Lemma~\ref{lemma:eigen-integrable}), with unit normal $E_1$. Moreover, notice that $\nabla_{E_i} E_1\perp E_1$, so these leaves are indeed spherical. Hence the metric decomposes as a warped product $I\times_f N$ (see \cite{Ponge-Twisted}). Moreover, since the Ricci eigenvalues $\lambda_i$ are equal for $i=2,3,4$, $N$ is Einstein, hence of constant sectional curvature. This implies that $I\times_f N$ is locally conformally flat (see \cite{Brozos-Loc-Conf-Flat}), so these solutions were already described in Theorem~\ref{th:loc-conf-flat-ndim}~(1), but not in Theorem~\ref{th:classification-diagonalizable}.

\subsection{Two eigenvalues coincide: $\lambda_2\neq \lambda_3= \lambda_4$ or $\lambda_2=\lambda_3\neq \lambda_4$.}

In order to fix notation, we adapt the pseudo-orthonormal basis $\{E_1=\nh/|\nh|, E_2,E_3,E_4\}$ so that $\lambda_2\neq \lambda_3= \lambda_4$ and we allow $g(E_i,E_i)=\varepsilon_i$ for $i=1,2,3$, so the unit timelike vector field could be $E_1$, $E_2$ or $E_3$.
In this context, the geometry of the manifold is so restricted that it decomposes as a multiply warped product.

\begin{lemma}\label{lemma:local-multi-warped}
Let $(M,g,h)$ be a 4-dimensional non-isotropic solution of the weighted Einstein field equations \eqref{eq:vacuum-Einstein-field-equations} with harmonic curvature tensor, such that the Ricci operator diagonalizes in the adapted local frame $\mathcal{B}_1=\{E_1,\dots,E_4\}$. If there are two distinct eigenvalues $\lambda_2\neq \lambda_3=\lambda_4$, then $(M,g)=I_1\times_\varphi I_2\times_\xi N$ is a multiply warped product manifold with metric 
\begin{equation}\label{eq:multiply-warped-decomposition}
	g=\varepsilon_1 dt^2+\varepsilon_2 \varphi(t)^2 ds^2+\xi(t)^2\tilde{g},
\end{equation}
where $\tilde{g}$ is the metric of a Riemannian or Lorentzian surface of constant Gauss curvature $\kappa$ and $h$ is a function on $t$.
\end{lemma}
\begin{proof}
We adapt the relations in \eqref{eq:connection-coefs} to this context to see that
\[
	\begin{array}{l}
		\nabla_{E_1}E_1=0, \; \nabla_{E_i}E_1=\beta_iE_i,\; \nabla_{E_1}E_2=0,\; \nabla_{E_1}E_3=\Gamma_{134} E_4,\;\nabla_{E_1}E_4=\varepsilon_3\Gamma_{143} E_3, \\
		\noalign{\medskip}
		\nabla_{E_2}E_2=-\varepsilon_1\varepsilon_2\beta_2E_1,\;  \nabla_{E_2}E_3=\Gamma_{23 4}E_4,\; \nabla_{E_2}E_4=\varepsilon_3\Gamma_{24 3}E_3 \\
		\noalign{\medskip}
		 \nabla_{E_3}E_3=-\varepsilon_1\varepsilon_3\beta_3 E_1+\Gamma_{334} E_4,\;
		\nabla_{E_4}E_4=-\varepsilon_1\beta_4 E_1+\varepsilon_3\Gamma_{443} E_3,\;\\
		\noalign{\medskip}
		\nabla_{E_3}E_4=\varepsilon_3\Gamma_{343} E_3,\;\nabla_{E_4}E_3=\Gamma_{434} E_4,\;
		\nabla_{E_3}E_2=\nabla_{E_4}E_2=0.
	\end{array}
\]
Notice that the tangent submanifolds to the distributions $\mathcal{D}_1=\spanned\{E_1,E_2\}$ and $\mathcal{D}_2=\spanned\{E_1,E_3,E_4\}$ are totally geodesic, since $\nabla \mathcal{D}_1\subset \mathcal{D}_1$ and $\nabla \mathcal{D}_2\subset \mathcal{D}_2$. We already know, by Lemma~\ref{lemma:eigen-integrable}, that leaves tangent to $\mathcal{D}_3=\spanned\{E_3,E_4\}$ are umbilical but, moreover, $\nabla_{E_i} E_1\perp E_1, E_2$  for $i\in\{3,4\}$, so they are indeed spherical. Since leaves tangent to $\spanned\{E_2\}$ are also spherical and we already have the decomposition in \eqref{eq:decomposition}, the manifold decomposes as a multiply warped product as in \eqref{eq:multiply-warped-decomposition} and, moreover, $h$ depends only on $t$.

Since $\lambda_3$ and $\lambda_4$ depend only on $t$ by Lemma~\ref{lemma:caseII-coordinate-s-eigenvalues}, the Ricci tensor on the fiber $\tilde g$ has constant eigenvalues and hence it is of constant Gauss curvature. 
\end{proof}

\begin{lemma}\label{lemma:mult-warping-constant}
	Let $(M,g,h)$ be a multiply warped product solution as in \eqref{eq:multiply-warped-decomposition} with $h=h(t)$ and harmonic curvature. Then one of the warping functions, either $\varphi$ or $\xi$, is constant.
\end{lemma}
\begin{proof}
Let $g$ be a multiply warped product metric as in \eqref{eq:multiply-warped-decomposition}. Let $\kappa$ be the Gauss curvature of $\tilde g$, we can choose local  coordinates $(x_2,x_3)$ on $N$ so that \[\tilde g(x_2,x_3)=\frac{1}{\left(1+\frac{\kappa}4 \left(\varepsilon_3 x_2^2+x_3^2\right)\right)^2} (\varepsilon_3 dx_2^2+dx_3^2).\] 
 Using these coordinates and introducing exponentials on the warping functions to simplify expressions, we set local coordinates $(t,x_1,x_2,x_3)$ on $M$ such that 
\begin{equation}\label{eq:local-coord-mult-warped}
	g=\varepsilon_1 dt^2+\varepsilon_2 e^{2 f_1(t)}dx_1^2 + e^{2 f_2(t)} \tilde g.
\end{equation}	
Now, from \eqref{eq:Rnf} we see that
\[
	\begin{array}{l}
	 0 =	R(\nabla h,\partial_{x_1},\partial_t,\partial_{x_1})-(\rho-2Jg)\wedge dh (\partial_{x_1},\partial_t,\partial_{x_1})=\\
		\noalign{\medskip}
		\,\;\qquad \frac{2 \varepsilon _1\varepsilon _2 e^{2 f_1} h' }{3}  \left(-\kappa 
		\varepsilon _1 e^{-2 f_2}-2 f_1'{}^2-f_1' f_2'+3
		f_2'{}^2-2 f_1''+2 f_2''\right),
	\end{array}
\]
from where
\[
e^{2 f_2} \left(2 \left(f_1'\right){}^2+f_1' f_2' -3 \left(f_2'\right){}^2+2 f_1''-2 f_2''\right)+\kappa  \varepsilon _1=0.
\] 
A direct computation also shows that
\[
\tau=-2 \varepsilon _1 \left(- \kappa  \varepsilon _1e^{-2 f_2}+\left(f_1'\right){}^2+2 f_1' f_2'+3 \left(f_2'\right){}^2+f_1''+2 f_2''\right).
\]
Using these two expressions we obtain that $f_1''=-\frac{1}{6} \left(6 \left(f_1'\right){}^2+6 f_1' f_2'+\tau  \varepsilon_1\right)$ and $f_2''=\frac{1}{2} e^{-2 f_2} \kappa  \varepsilon _1-\frac{3}{2} \left(f_2'\right){}^2-\frac{1}{2} f_1' f_2'-\frac{\tau  \varepsilon _1}{6}$.

Moreover, we compute
\[
G^h(\partial_{x_1},\partial_{x_1})=\frac{1}{6} e^{2 f_1} \varepsilon _1 \varepsilon _2 \left(6 \left(2 f_2' h'+h''\right)+h \left(\tau  \varepsilon _1-6 f_1' f_2'\right)\right)
\]
to get that $h''=\frac{1}{6} \left(-12 f_2' h'+6 h f_1' f_2'-h \tau  \varepsilon _1\right)$. Thus, we have expressed the second derivatives of $h$, $f_1$ and $f_2$ in terms of lower order terms.
Now, we use these relations to compute
\[
\begin{array}{rcl}
G^h(\partial_t,\partial_t)&=& \left(f_1'+2 f_2'\right) h'+ \left(- \kappa  \varepsilon _1e^{-2 f_2}+\left(f_2'\right){}^2+2 f_1' f_2'+\frac{\tau  \varepsilon _1}{2}\right)h,\\
\noalign{\medskip}
G^h(\partial_{x_3},\partial_{x_3})&=& \frac{8\varepsilon_1e^{2 f_2}}{\left(\kappa  x_2^2 \varepsilon_3+
							\kappa  x_3^2+4\right){}^2}\left( \left( f_1' f_2'- \left(f_2'\right){}^2+\kappa 
\varepsilon _1e^{-2 f_2}\right)h+2  \left(f_1'-f_2'\right) h'  \right).
\end{array}
\]
This leads to a homogeneous linear system of two equations in the unknowns $h'$ and $h$. Hence, the rank of the associated matrix is zero:
\begin{equation}\label{eq:C1_rank_matrix}
C_1= e^{2 f_2} \left(f_1'-f_2'\right) \left(3 f_1' f_2'+\tau  \varepsilon _1\right)-3 \kappa  \varepsilon _1 f_1'=0.
\end{equation}
Differentiating with respect to $t$ and simplifying second order terms, we get 
\[
 f_1' \left(3 \kappa  \varepsilon _1 f_1'-e^{2 f_2} \left(f_1'-f_2'\right) \left(4 \left(f_2'\right){}^2+5 f_1' f_2'+\tau  \varepsilon
_1\right)\right)=0.
\]
Thus, either $f_1'=0$, in which case the lemma holds, or \[C_2=3\kappa \varepsilon_1 f_1'-e^{2 f_2} \left(f_1'-f_2'\right) \left(4 \left(f_2'\right){}^2+5 f_1' f_2'+\tau  \varepsilon_1\right) =0.\] 
If $C_2=0$, then we compute
\[
0=C_1+C_2=-2 e^{2 f_2} f_2' \left(f_1'-f_2'\right) \left(f_1'+2 f_2'\right).
\]
Hence, there are three possible cases. If $f_2'=0$, the result follows. If $f_1'=f_2'$, then by \eqref{eq:C1_rank_matrix}, we have either $f_1'=0$, so the result follows, or $\kappa=0$, in which case the manifold is locally conformally flat. Finally, if $f_1'=-2 f_2'$, then $f_1(t)=C-2f_2(t)$ for a constant $C$, and $G^h(\partial_t,\partial_t)=-6 h (f_2')^2$, so $f_2'=0$.
\end{proof}

Lemma~\ref{lemma:mult-warping-constant} above shows that at least one of the warping functions is constant. Notice that if both are constant, then we have a direct product. In this case, a direct computation of the equation $G^h=0$ shows that, necessarily, $\kappa=0$, so the manifold is flat.
As a result, we can restrict our analysis of the multiply warped product solutions to the case where one of the warping functions is constant and the other is strictly non-constant. We first analyze the case $\xi'\neq 0$.

\subsection*{Case $\varphi=\text{constant}$}

\begin{lemma}\label{lemma:local-multi-warped-sol-1}
Let $(M,g,h)$ be a multiply warped product solution as in \eqref{eq:multiply-warped-decomposition} with  $\varphi$ constant, with harmonic curvature and $h=h(t)$. Then, $(M,g,h)$ is a solution of \eqref{eq:vacuum-Einstein-field-equations} if and only if $\tau=0$, $h=c \xi'$ and $I_1 \times_\xi N$ is one of the 3-dimensional locally conformally flat solutions portrayed in Example~\ref{ex:3-dim-loc-conf-flat-sols} for the density function $h$.
\end{lemma}
\begin{proof}
Consider a multiply warped product structure as in \eqref{eq:multiply-warped-decomposition} with $\varphi$ constant. Normalize the coordinate $s$ if necessary so that $\varphi=1$. 
Because of the metric structure we have $R(\partial_t,\partial_s,\partial_t,\partial_s)=0$ and $\rho(\partial_s,\partial_s)=0$, so we obtain $0=\tau=\frac{2 \kappa  -\varepsilon_1 4 \xi\xi ''-2\varepsilon_1 (\xi ')^2}{ \xi^2}$ from \eqref{eq:Rnf}. This implies that the scalar curvature of $I_1\times_\xi N$ also vanishes.

From equation \eqref{eq:vacuum-Einstein-field-equations}, we have that
\[
G^h(\partial_t,\partial_t)=2\frac{ h' \xi'- h \xi ''}{\xi}, \text{ and } G^h(\partial_s,\partial_s)=\frac{\varepsilon_1\varepsilon_2 \left(\xi h''+2 h' \xi '\right)}{ \xi }.
\]
Hence, on the one hand, solving $h'\xi'-h\xi''=0$ we get $h=c \xi'$ for a constant $c\neq 0$. On the other hand,
since $h'\xi'-h\xi''=0$, we write $0=\varepsilon_1\varepsilon_2 G^h(\partial_s,\partial_s)=h''+2\frac{\xi'h'}{\xi}=h''+2h\frac{\xi''}\xi$  to obtain the system of ODEs \eqref{eq:weighted-field-ODEs} for a locally conformally flat $3$-dimensional solution with vanishing scalar curvature, which are further  described by Example~\ref{ex:3-dim-loc-conf-flat-sols}. 

Conversely, take any $3$-dimensional solution $(I_1\times_\xi N,h)$, with $(N,g^N)$ of constant Gauss curvature $\kappa$, $h=h(t)$, and vanishing scalar curvature. Consider the $4$-dimensional manifold $(M,g)=I_1\times I_2 \times_\xi N$. Then, $\tau=\frac{2 \kappa  -4\varepsilon_1 \xi\xi ''-2\varepsilon_1 (\xi ')^2}{\xi^2}=0$. Moreover,
because the system of ODEs \eqref{eq:weighted-field-ODEs} is satisfied, we have  $h' \xi'- h \xi ''=0$ and $\xi h''+2 h' \xi '=0$, which imply $G^h(\partial_t,\partial_t)=G^h(\partial_s,\partial_s)=0$. Using that $h=c \xi'$ we compute: 
\[
\begin{array}{rcl}
\varepsilon_1  \xi^2 G^h(X,X)&=&\left(\xi  \left(h' \xi '+\xi  h''\right)+h \left(\kappa  \varepsilon _1-\xi  \xi ''-\left(\xi'\right)^2\right)\right)g(X,X)\\
\noalign{\medskip}
&=&c (\kappa  \varepsilon_1  \xi
'-\left(\xi '\right)^3+\xi ^2 \xi ^{(3)})g(X,X)
\end{array}
\]
for  any vector $X$ tangent to $N$. Since $\tau'=-\frac{4 (\kappa  \xi '-\varepsilon _1 \left(\xi '\right)^3+\varepsilon _1\xi ^2 \xi ^{(3)})
	}{\xi ^3}=0$, this term vanishes, so $(M,g,h)$ is a $4$-dimensional solution. Moreover, $\kappa  \xi '-\varepsilon _1 \left(\xi '\right)^3+\xi ^2 \xi ^{(3)}
	\varepsilon _1=0$ is also the necessary and sufficient condition for the manifold to have harmonic curvature. 
\end{proof}

\begin{remark}
	Note that solutions given in Lemma~\ref{lemma:local-multi-warped-sol-1} generically present three distinct eigenvalues for $\Ric$. Indeed, there is a zero eigenvalue corresponding to the $I_2$ factor and the number of eigenvalues reduces to two only if $I_1\times_\xi N$ is Einstein, in which case the underlying manifold is flat.
\end{remark}

\subsection*{Case $\xi=\text{constant}$}

\begin{lemma}\label{lemma:local-multi-warped-sol-2}
Let $(M,g,h)$ be a multiply warped product solution as in \eqref{eq:multiply-warped-decomposition}, with $\xi$ constant, with harmonic curvature and  $h=h(t)$. Then $(M,g)$ is a direct product of two surfaces of constant Gauss curvature and $h=c \varphi'$.
\end{lemma}
\begin{proof}
Since the metric  $\xi^2\tilde{g}$ on $N$ has constant Gauss curvature $\frac{\kappa}{\xi^2}$, we can assume $\xi=1$ in \eqref{eq:multiply-warped-decomposition} by a change of coordinates and a redefinition of $\kappa$. Then, the Ricci operator is given by
\[
\Ric(\partial_t)=-\frac{\varepsilon_1\varphi''}{\varphi}\partial_t, \quad Ric(\partial_s)=-\frac{\varepsilon_1\varphi''}{\varphi}\partial_s,\quad \Ric(X)=\kappa X \text{ for } X \text{ tangent to } N.
\]
From \eqref{eq:Rnf}, we have 
\[
	0=R(\nabla h,\partial_s,\partial_t,\partial_s)-(\rho-2Jg)\wedge dh(\partial_s,\partial_t,\partial_s)=- \frac{2}{3}\varepsilon_2 \varphi  h' \left(\kappa    \varphi +2\varepsilon_1 \varphi ''\right).
\]
This implies that $-\frac{\varepsilon_1\varphi''}{\varphi}=\frac{\kappa}2$, so the manifold is a direct product of two surfaces $N_1$ and $N_2$ with constant Gauss curvatures $\frac{\kappa}2$ and $\kappa$ respectively. Moreover,  we compute $0=G^h(\partial_t,\partial_t)=\frac{h'\varphi'-h\varphi''}{\varphi}$, to see that $h=c\varphi'$ for a suitable integration constant $c\in \mathbb{R}^*$.
\end{proof}

\begin{remark} 
 Assume we are in the conditions of Lemma~\ref{lemma:local-multi-warped-sol-2}. Then, we can use the product structure of \eqref{eq:multiply-warped-decomposition} with $\xi=1$ to write simple coordinate expressions for $\varphi$ and $h$. 
  Firstly, note that $\kappa$ cannot vanish, since this results in a direct product of two flat surfaces (hence a flat solution), with constant density function $h$.

Now, for $\kappa\neq 0$, since $-\frac{\varepsilon_1\varphi''}{\varphi}=\frac{\kappa}2$, the warping function $\varphi$ takes the following forms, depending on the sign of the product $\varepsilon_1\kappa$:
	\[
	\begin{array}{rclr}
		\varphi(t)&=&c_1\sin\left(\sqrt{\frac{\varepsilon_1\kappa}{2}}t\right)+c_2\cos\left(\sqrt{\frac{\varepsilon_1\kappa}{2}}t\right), \,&\text{ if } \varepsilon_1\kappa >0, \\
		\noalign{\medskip}
		\varphi(t)&=&c_1e^{\sqrt{-\frac{\varepsilon_1\kappa}{2}}t}++c_1e^{-\sqrt{-\frac{\varepsilon_1\kappa}{2}}t}, \,&\text{ if } \varepsilon_1\kappa<0,
	\end{array}
	\]
	where $c_1$ and $c_2$ are suitable integration constants so that $\varphi(t),h(t)>0$ for all $t\in I$.
\end{remark}

\subsection{Case 3. The eigenvalues of the Ricci operator are different ($\lambda_2\neq \lambda_3\neq \lambda_4$)}

There are no solutions with three different warping functions. The proof follows the main steps to that given in \cite{Kim-Shin}. Due to its length, in order not to break the argumentative flow, we postpone it to Appendix~\ref{appendix}.

\subsection{Proof of Theorem~\ref{th:classification-diagonalizable}~(2).}
Let $(M,g,h)$ be a $4$-dimensional solution with diagonalizable Ricci operator and harmonic curvature (not locally conformally flat), and such that $g(\nh,\nh)\neq 0$. Then, for any point in $M_{\Ric}$, which is open and dense in $M$, we apply Lemmas~\ref{lemma:nh-Ricci-eigen}-\ref{lemma:local-multi-warped-sol-2} to find the admissible structures at the local level. Note that, in the hypotheses of Lemma~\ref{lemma:local-multi-warped-sol-2}, the fact that $h$ satisfies the Obata equation on $N_1$ is immediate from $\Hes^{N_1}_h=\restr{\Hes_h}{N_1}=h\restr{(\rho-\frac{\tau}{3}g)}{N_1}=-\frac{\kappa h}2 g^{N_1}$ (since $\tau=3\kappa$). \qed


\section{The case with complex eigenvalues}\label{sec:non-diagonalizable-complex}

Throughout the previous section, we have discussed the admissible solutions with diagonalizable Ricci operator (the only possible case if $\nh$ is timelike), and we have seen that the geometric characteristics of such solutions are not dissimilar from those of Riemannian static spaces discussed in \cite{Kim-Shin, Kobayashi,Qing-Yuan} if the solution is non-isotropic. However, the fact that we are working in Lorentzian signature means that $\Ric$ does not diagonalize, in general, when $\nh$ is spacelike or lightlike. For isotropic solutions, it was shown in \cite{Brozos-Mojon} that the eigenvalues of the Ricci operator are necessarily real. In this section, we show that this is also the case for 4-dimensional non-isotropic solutions with harmonic curvature.


\begin{theorem}\label{th:non-existence-complex}
	Let $(M,g,h)$ be a 4-dimensional solution of the vacuum weighted Einstein field equations \eqref{eq:vacuum-Einstein-field-equations}  with harmonic curvature. Then, the Ricci operator of $(M,g)$ has real eigenvalues.
\end{theorem}

Note that the harmonicity of the curvature is an essential assumption in Theorem~\ref{th:non-existence-complex}, since there are solutions with non-real eigenvalues and non-harmonic curvature, as illustrated by the following example.

\begin{example}\label{ex:complex-eigenvalues}
In order to build a solution with complex eigenvalues for the Ricci operator, we consider a left-invariant metric on the Lie group $\mathbb{R}^3\rtimes\mathbb{R}$, this is, a semi-direct extension of the Abelian group. Let $\{e_1,e_2,e_3,e_4\}$ be a basis of the corresponding Lie algebra, where $e_4$ generates the $\mathbb{R}$ factor, and the Lie bracket given by
\[
[e_1,e_4]=-e_1,\quad [e_2,e_4]=e_3,\quad [e_3,e_4]=-e_2.
\]
The Lorentzian metric is given by $\langle e_1,e_1\rangle=\langle e_2,e_2\rangle=-\langle e_3,e_3\rangle=\langle e_4,e_4\rangle=1$. 

Now, we look for an expression of the metric in local coordinates $(x,y,z,t)\in \mathbb{R}^4$.
Using the relation $d\omega(X,Y)=X \omega(Y)-Y\omega(X)-\omega([X,Y])$, for a $1$-form $\omega$, on the dual basis $\{e^1,e^2,e^3,e^4\}$ we obtain
\[
de^1=e^1\wedge e^4,\quad de^2=e^3\wedge e^4, \quad de^3=-e^2\wedge e^4,\quad de^4=0.
\]
By relating the basis $\{e^1,e^2,e^3,e^4\}$ with $\{dx,dy,dz,dt\}$, and integrating the corresponding equations we get a particular solution of the form  $e^1=e^{-t} dx$, $e^2=\cos t dy+\sin t dz$, $e^3=\sin t dy- \cos(t) dz$, $e^4=dt$.
Hence,
	\[
		g=e^{2t}dx^2+\cos(2t)(dy^2-dz^2)+2\sin(2t)dydz+dt^2.
	\]
	Consider the positive density function $h(t)=e^{-t}$, whose gradient is $\nh=-e^{-t}\partial_t$, so $g(\nh,\nh)=e^{-2t}>0$. Moreover, the Ricci and the Hessian operators are given by
	\[
	\begin{array}{l}
		  \Ric(\partial_x)=\partial_t, \quad \Ric(\partial_y)=\partial_z, \quad \Ric(\partial_z)=-\partial_y, \quad
		  \Ric(\partial_t)=\partial_t,\\
		  \noalign{\medskip}
		   \hes_h(\partial_x)=e^{-t}\partial_t, \quad \hes_h(\partial_y)=e^{-t}\partial_z, \quad \hes_h(\partial_z)=-e^{-t}\partial_y, \quad
		  \hes_h(\partial_t)=e^{-t}\partial_t,
	\end{array}
	\]
	so $(\mathbb{R}^4,g,h)$ is a solution of Type I.b with $\lambda=1$, $\alpha=b=-1$ and $a=0$. However, the curvature tensor is not harmonic, since
	\[
		(\nabla_{\partial_x}\rho)(\partial_t,\partial_t)-(\nabla_{\partial_t}\rho)(\partial_x,\partial_t)=2e^{2t}\neq 0.
	\]
\end{example}

In order to prove Theorem~\ref{th:non-existence-complex}, assume, on the contrary, that $(M,g,h)$ is a 4-dimensional solution of the vacuum weighted Einstein field equations  with $\nh$ spacelike and Ricci operator of Type I.b in $M$, as shown in \eqref{eq:matrices-type-IaandIb}. We work on an adapted orthonormal basis $\mathcal{B}_1=\{E_1=\nh/ |\nh|, E_2,E_3,E_4\}$ and see that, by the weighted Einstein equation \eqref{eq:vacuum-Einstein-field-equations-2}, the Hessian operator is given by 
\[
\hes_h\nh=\tilde{\lambda}\nh, \quad \hes_hE_2=\tilde{a}E_2-\tilde{b}E_3, \quad \hes_hE_3=\tilde{b}E_2+\tilde{a}E_3, \quad \hes_hE_4=\tilde{\alpha}E_4,
\]
where $\tilde{\lambda}=h(\lambda-2J)$, $\tilde{a}=h(a-2J)$, $\tilde{b}=hb$ and $\tilde{\alpha}=h(\alpha-2J)$.
We start by using the harmonicity of the curvature to obtain information on $a$, $b$ $\lambda$ and $\alpha$, and on the components of the Levi-Civita connection, with the following two lemmas.

\begin{lemma}\label{le:complex-derivatives}
	Let $(M,g,h)$ be a solution of Type I.b such that $(M,g)$ has harmonic curvature. Then, $a$, $b$, $\lambda$ and $\alpha$ have vanishing derivatives in the direction of $E_2$, $E_3$ and $E_4$. Moreover,
	\begin{equation}\label{eq:complex-eigen-derivatives}
		\begin{array}{rcl}
			\nh(a)&=&h(b^2+(\lambda-a)(a-2J)), \\
			\noalign{\medskip}
			\nh(b)&=&hb(\lambda+2J-2a), \\
			\noalign{\medskip}
			\nh(\alpha)&=&h(\lambda-\alpha)(\alpha-2J), \\
			\noalign{\medskip}
			\nh(\lambda)&=&-h(2b^2-2a^2-\alpha^2+(\lambda+2J)(2a+\alpha)-6J\lambda).
		\end{array}
	\end{equation}
\end{lemma}
\begin{proof}
In the frame $\mathcal{B}_1$, and using the weighted Einstein equation \eqref{eq:vacuum-Einstein-field-equations-2}, we compute $(\nabla_{\nh}\rho)(E_2,\nh)=0$ and $(\nabla_{E_2}\rho)(\nh,\nh)=|\nh|^2E_2(\lambda)$. Since $\nh$ is spacelike, from the Codazzi condition $(\nabla_{\nh}\rho)(E_2,\nh)=(\nabla_{E_2}\rho)(\nh,\nh)$, we find $E_2(\lambda)=0$. Similarly, we prove $E_3(\lambda)=E_4(\lambda)=0$. 

We continue to obtain information from the Codazzi condition. On the one hand we have $(\nabla_{\nh}\rho)(E_2,E_3)=-\nh(b)$ and $(\nabla_{E_2}\rho)(\nh,E_3)=-(\lambda-a)\tilde{b}+b\tilde{a}$, which gives the expression for $\nh(b)$. In the same way, we compute $(\nabla_{\nh}\rho)(E_i,E_i)$ and $(\nabla_{E_i}\rho)(\nh,E_i)$ for $i=2,3$ to find
\[
\begin{array}{rcl}
	-\nh(a)+2bg(\nabla_{\nh}E_2,E_3)=(a-\lambda)\tilde{a}-b\tilde{b}, \\
	\noalign{\medskip}
	\nh(a)+2bg(\nabla_{\nh}E_2,E_3)=(\lambda-a)\tilde{a}+b\tilde{b},
\end{array}
\]
which yields $g(\nabla_{\nh}E_2,E_3)=0$ and the expression for $\nh(a)$. On the other hand, $(\nabla_{\nh}\rho)(E_4,E_4)=\nh(\alpha)$ and $(\nabla_{E_4}\rho)(\nh,E_4)=(\lambda-\alpha)\tilde{\alpha}$ gives us an equation for $\nh(\alpha)$.  Finally, since $\tau$ is constant by Lemma~\ref{le:const-sc}, and $\tau=\lambda+2a+\alpha$, we have $0=\nh(\lambda)+2\nh(a)+\nh(\alpha)$, which yields the last equation in \eqref{eq:complex-eigen-derivatives}. Moreover, since we know $E_i(\lambda)=0$ for $i=2,3,4$ and $\tau$ is constant, it follows that $2E_i(a)=-E_i(\alpha)$. Thus, taking the derivative of the equation above in the direction of $E_i$, for $i=2,3,4$, it follows that  $E_i(2a^2+\alpha^2-2b^2)=0$. Now, take
\[
\begin{array}{rcl}
	h^{-1}\nh(a^2+\frac{1}2\alpha^2-b^2)&=&(\lambda+2J)(2a^2+\alpha^2-2b^2)-2J\lambda(2a+\alpha)\\
	\noalign{\medskip}
	&&-2a^3-\alpha^3+6ab^2.
\end{array}
\]
Differentiating this expression in the direction of $E_i$ yields $E_i(2a^3+\alpha^3-6ab^2)=0$. In summary, we have three distinct expressions:
\[
	E_i(\alpha)=-2E_i(a), \quad E_i(2a^2+\alpha^2-2b^2)=0, \quad E_i(2a^3+\alpha^3-6ab^2)=0,
\]
for $i=2,3,4$. Using the first and second ones, we can write $E_i(b)=\frac{a-\alpha}b E_i(a)$,
so now the third equation becomes 
\[
\begin{array}{rcl}
	0&=&6a^2E_i(a)+3\alpha^2E_i(\alpha)-6b^2E_i(a)-12abE_i(b) \\
	\noalign{\medskip}
	&=&-6((a-\alpha)^2+b^2)E_i(a).
\end{array}
\]
Since $b\neq 0$, it follows that $E_i(a)=E_i(\alpha)=E_i(b)=0$.
\end{proof}

\begin{lemma}\label{le:connection-complex}
Let $(M,g,h)$ be a solution of Type I.b such that $(M,g)$ has harmonic curvature. Let $C$ be the matrix such that $C_{1i}=\nabla_{\nh}E_i$,  $C_{i1}=\nabla_{E_i}\nh$ and $C_{ij}=\nabla_{E_i}E_j$, for $i,j\in\{2,3,4\}$. Then,
\[
	C=\begin{pmatrix}
		\tilde{\lambda}\nh & 0 & 0 & 0 \\
		\tilde{a} E_2-\tilde{b}E_3 & \frac{\tilde{a}}{|\nh|^2}\nh+\frac{\alpha-a}b \Gamma E_4 & \frac{\tilde{b}}{|\nh|^2}\nh+\Gamma E_4& \Gamma\left(\frac{\alpha-a}b  E_2- E_3\right) \\
		\tilde{b} E_2+\tilde{a}E_3  & \frac{\tilde{b}}{|\nh|^2}\nh-\Gamma E_4 &-\frac{\tilde{a}}{|\nh|^2}\nh+\frac{\alpha-a}b \Gamma E_4  & -\Gamma\left( E_2+\frac{\alpha-a}b E_3\right) \\
		\tilde{\alpha} E_4   & -\frac{(\alpha-a)^2+b^2}{2b^2} \Gamma E_3& -\frac{(\alpha-a)^2+b^2}{2b^2} \Gamma E_2 &-\frac{\tilde{\alpha}}{|\nh|^2}\nh \\
	\end{pmatrix}
\]
where $\Gamma=g(\nabla_{E_2}E_3,E_4)$.
\end{lemma}
\begin{proof}
	The column $C_{i1}$ is given by the weighted Einstein equation \eqref{eq:vacuum-Einstein-field-equations-2} and the fact that $\nabla_{E_i}\nh=\hes_hE_i$. We also use $g(\nabla_{E_i}E_j,\nh)=-\Hes_h(E_i,E_j)$ to find the component in the direction of $\nh$ of $\nabla_{E_i}E_j$. Now, from the proof of Lemma~\ref{le:complex-derivatives}, we know that $g(\nabla_{\nh}E_2,E_3)=-g(\nabla_{\nh}E_3,E_2)=0$. Next, we compute $(\nabla_{\nh}\rho)(E_i,E_4)=(\nabla_{E_i}\rho)(\nh,E_4)$, with $i=2,3$, to find
	\[
	\begin{array}{rcl}
		(\alpha-a)g(E_2,\nabla_{\nh}E_4)+bg(E_3,\nabla_{\nh}E_4)&=&0, \\
		\noalign{\medskip}
		(\alpha-a)g(E_3,\nabla_{\nh}E_4)-bg(E_2,\nabla_{\nh}E_4)&=&0.
	\end{array}
	\]
	Since $b\neq 0$, we have $g(\nabla_{\nh}E_i,E_4)=-g(\nabla_{\nh}E_4,E_i)=0$, for $i=2,3,4$. Moreover, $g(\nabla_{\nh}E_i,\nh)=\Hes_h(E_i,\nh)=0$ for $i=2,3,4$. This completes the row $C_{1i}$. 
	
	Let $\Gamma_{ijk}=g(\nabla_{E_i}E_j,E_k)$ (notice that $\Gamma_{ijk}=-\Gamma_{ikj}$) and $\nabla_i\rho_{jk}=(\nabla_{E_i}\rho)(E_j,E_k)$. Then, compute $\nabla_2\rho_{33}=2b\Gamma_{223}$ and $\nabla_3\rho_{23}=0$ to find $\Gamma_{223}=0$. Analogously, from $\nabla_3\rho_{22}=\nabla_2\rho_{32}$, we have $\Gamma_{332}=0$. Moreover, from $\nabla_i\rho_{44}=\nabla_4\rho_{i4}$ for $i=2,3$ it follows that
	\[
		0=b\Gamma_{443}+(\alpha-a)\Gamma_{442}, \qquad 0=(\alpha-a)\Gamma_{443}-b\Gamma_{442},
	\]
	from where $\Gamma_{443}=\Gamma_{442}=0$. Hence, the only non-vanishing $\Gamma_{ijk}$ (up to symmetries) are $\Gamma_{4ij}$ and $\Gamma_{ij4}$, where $i,j=2,3$. Finally, we use $\nabla_4\rho_{ii}=\nabla_i\rho_{4i}$ to find
	\[
	2b\Gamma_{423}=-(\alpha-a)\Gamma_{224}-b\Gamma_{234}, \qquad 2b\Gamma_{423}=-(\alpha-a)\Gamma_{334}+b\Gamma_{324},
\]
while $\nabla_3\rho_{24}=\nabla_2\rho_{34}=\nabla_4\rho_{23}$ gives two more relations:
\[
	0=(\alpha-a)\Gamma_{234}-b\Gamma_{224}, \qquad 0=(\alpha-a)\Gamma_{324}+b\Gamma_{334}.
\]
Setting $\Gamma_{234}=\Gamma$, the indeterminate system given by these four equations yields
\[
		\Gamma_{224}=\Gamma_{334}=\frac{\alpha-a}b \Gamma, \qquad \Gamma_{324}=-\Gamma, \qquad \Gamma_{423}=-\frac{(\alpha-a)^2+b^2}{2b^2} \Gamma,
\]
which completes the remaining terms of the matrix $C$.
\end{proof}

The two lemmas above exhaust the amount of information we can extract from the harmonicity of the curvature tensor. However, more compatibility conditions can be obtained through the Jacobi identity of vector fields and the restrictions that the vacuum weighted Einstein equations impose on the curvature tensor.

\begin{lemma}\label{le:Eigamma}
	Let $(M,g,h)$ be a solution of Type I.b of the weighted Einstein equations such that $(M,g)$ has harmonic curvature, with the connection coefficients given by Lemma~\ref{le:connection-complex}. Then, $E_i(\Gamma)=0$ for $i=2,3,4$.
\end{lemma}
\begin{proof}
	We use the Jacobi identity of vector fields to write
\[
	[[E_4,E_2],E_3]+[[E_2,E_3],E_4]+[[E_3,E_4],E_2]=0.
\]
Using the notation of Lemma~\ref{le:connection-complex}, the Lie brackets take the form
\[
\begin{array}{rcl}
	[E_4,E_2]&=&-\frac{\alpha-a}b \Gamma E_2-\frac{(\alpha-a)^2-b^2}{2b^2} \Gamma E_3, \qquad [E_2,E_3]=2\Gamma E_4,\\
	\noalign{\medskip}
	[E_3,E_4]&=& -\frac{\alpha-a}b\Gamma E_3+\frac{(\alpha-a)^2-b^2}{2b^2} \Gamma E_2.
\end{array}
\]
Moreover, by Lemma~\ref{le:complex-derivatives}, we have that $E_i(\alpha)=E_i(a)=E_i(b)=0$ for $i=2,3,4$. Hence,
\[
\begin{array}{rcl}
	[[E_4,E_2],E_3]&=&\frac{\alpha-a}b E_3(\Gamma) E_2+\frac{(\alpha-a)^2-b^2}{2b^2} E_3(\Gamma) E_3-2\frac{\alpha-a}b \Gamma^2E_4,  \\
	\noalign{\medskip}
	[[E_2,E_3],E_4]&=&-2E_4(\Gamma)E_4, \\
	\noalign{\medskip}
	[[E_3,E_4],E_2]&=& 2\frac{\alpha-a}b\Gamma^2 E_4+\frac{\alpha-a}bE_2(\Gamma) E_3-\frac{(\alpha-a)^2-b^2}{2b^2} E_2(\Gamma) E_2.
\end{array}
\]
Taking the sum of the three brackets, it follows that $E_4(\Gamma)=0$. Moreover, from the components in the direction of $E_2$ and $E_3$ respectively, we have
\[
	0=\frac{\alpha-a}b E_3(\Gamma) -\frac{(\alpha-a)^2-b^2}{2b^2} E_2(\Gamma), \quad 0=\frac{(\alpha-a)^2-b^2}{2b^2} E_3(\Gamma)+\frac{\alpha-a}bE_2(\Gamma).
\]
The determinant associated to this homogeneous system is  $\left(\frac{(\alpha-a)^2+b^2}{2b^2}\right)^2\neq 0$, so the only solution is $E_2(\Gamma)=E_3(\Gamma)=0$. 
\end{proof}

\begin{lemma}\label{le:eqs-gamma-complex}
	Let $(M,g,h)$ be a solution of Type I.b of the weighted Einstein equations such that $(M,g)$ has harmonic curvature, with the connection coefficients given by Lemma~\ref{le:connection-complex}. Then, the following equations are satisfied:
\begin{align}
	0&=\tfrac{(\alpha-a)^2+b^2}{b^2}\Gamma^2-2\left(a-J\right)-\tfrac{h^2}{|\nh|^2}(\left(a-2J\right)^2+b^2+\left(a-2J\right)\left(\alpha-2J\right)), \label{eq:complexsist1} \\
	0&=\tfrac{(\alpha-a)((\alpha-a)^2+b^2)}{b^2}\Gamma^2+2b^2+\tfrac{h^2}{|\nh|^2}b^2\left(\alpha-2J\right),  \label{eq:complexsist2}  \\
	0&=\tfrac{(\alpha-a)^2+b^2}{b^2}\Gamma^2+\alpha-J+\tfrac{h^2}{|\nh|^2}\left(a-2J\right)\left(\alpha-2J\right).  \label{eq:complexsist3}
\end{align}
\end{lemma}
\begin{proof}
Recall that $E_1=\frac{\nh}{|\nh|}$ and denote \[R_{ijkl}=R(E_i,E_j,E_k,E_l)=g((\nabla_{[E_i,E_j]}-[\nabla_{E_i},\nabla_{E_j}])E_k,E_l).\] We will use Lemmas~\ref{le:complex-derivatives}, \ref{le:connection-complex} and \ref{le:Eigamma} to compute the different components of the curvature tensor. For example,
\[
\begin{array}{rcl}
	R(E_2,E_4)E_3&=&\frac{\alpha-a}{b}\Gamma \nabla_{E_2}E_3+\frac{(\alpha-a)^2-b^2}{2b^2}\Gamma\nabla_{E_3}E_3  \\
	\noalign{\medskip}
	&&+\frac{(\alpha-a)^2+b^2}{2b^2}\Gamma\nabla_{E_2}E_2+\frac{\tilde{b}}{|\nh|^2}\nabla_{E_4}\nh+\Gamma\nabla_{E_4}E_4 \\
	\noalign{\medskip}
	&=&\left(\frac{(\alpha-a)((\alpha-a)^2+b^2)}{b^3}\Gamma^2+\frac{\tilde{b}\tilde{\alpha}}{|\nh|^2}\right)E_4.
\end{array}
\]
Other components follow analogously, and the following are of interest:
\begin{equation}\label{eq:curv-comp-complex}
\begin{array}{rcl}
	-R_{3434}=R_{2424}&=&\frac{(\alpha-a)^2+b^2}{b^2}\Gamma^2+\frac{h^2}{|\nh|^2}\left(a-2J\right)\left(\alpha-2J\right), \\
	\noalign{\medskip}
	R_{2434}&=&\frac{(\alpha-a)((\alpha-a)^2+b^2)}{b^3}\Gamma^2+\frac{h^2}{|\nh|^2}b\left(\alpha-2J\right),	\\
	\noalign{\medskip}
	R_{2323}&=&-2\frac{(\alpha-a)^2+b^2}{b^2}\Gamma^2+\frac{h^2}{|\nh|^2}\left(\left(a-2J\right)^2+b^2\right).
\end{array}
\end{equation}
On the other hand, we use equation~\eqref{eq:Rnf} to compute the following components involving $E_1$:
\begin{equation}\label{eq:R1ijk-complex}
	R_{2121}=a-2J, \qquad R_{2131}=b, \qquad R_{4141}=-\alpha+2J.
\end{equation}
Now, using the definition of the Ricci tensor, we have
\[
\begin{array}{rcl}
	-a&=&\rho_{22}=R_{2121}+R_{2323}+R_{2424}, \\
	\noalign{\medskip}
	-b&=&\rho_{23}=R_{2131}+R_{2434},\\
	\noalign{\medskip}
	\alpha&=&\rho_{44}=R_{4141}-R_{2424}+R_{3434}.
\end{array}
\]
Substituting in the expressions given by \eqref{eq:curv-comp-complex} and \eqref{eq:R1ijk-complex}, the result follows.
%
%
\end{proof}

\subsection{Proof of Theorem~\ref{th:non-existence-complex}}
	Let $(M,g,h)$ be a 4-dimensional Type I.b solution of the vacuum weighted Einstein field equations \eqref{eq:vacuum-Einstein-field-equations} such that $(M,g)$ has harmonic curvature. For this solution, Lemmas~\ref{le:complex-derivatives}-\ref{le:eqs-gamma-complex} stated throughout this section apply. Let $H=\frac{h^2}{|\nh|^2}$. We analyze two cases separately: $\alpha=a$ and $\alpha \neq a$. 
	
	\vspace{1em}
	
	\noindent\textbf{Case $\alpha=a$:} Equations~\eqref{eq:complexsist1}, \eqref{eq:complexsist2} and \eqref{eq:complexsist3} in Lemma~\ref{le:eqs-gamma-complex} reduce to
		\[
\begin{array}{rcl}
	0&=&2\left(a-J\right)+H(2\left(a-2J\right)^2+b^2)-\Gamma^2, \\
	\noalign{\medskip}
	0&=&2b^2+Hb^2\left(a-2J\right),  \\
	\noalign{\medskip}
	0&=&a-J+H\left(a-2J\right)^2+\Gamma^2.
\end{array}
\]
Since $b\neq 0$, we solve for $a$ in the second expression to get $a=2\frac{JH-1}{H}$. The remaining two equations become
		\[
	0=b^2H+\frac{4}{H}+2J-\Gamma^2, \qquad 0=\frac{2}{H}+J+\Gamma^2,
\]
so adding both yields
\begin{equation}\label{eq:lambda=a-eq1}
	0=b^2H+\frac{6}{H}+3J.
\end{equation}
Now, notice that $\nh(H)=2h(1-H(\lambda-2J))$. Using this, the fact that $\lambda=6J-3a=\frac{6}{H}$, and the expression for $\nh(b)$ given by Lemma~\ref{le:complex-derivatives}, we differentiate \eqref{eq:lambda=a-eq1} in the direction of $\nh$. This gives $0=2\frac{h}{H^2}(5b^2H^2-12JH+30)$. Hence,
\begin{equation}\label{eq:lambda=a-eq2}
	0=5b^2H^2-12JH+30.
\end{equation}
Combining \eqref{eq:lambda=a-eq1} and \eqref{eq:lambda=a-eq2}, it follows that $0=6+b^2H^2$, which is not possible.

\vspace{1em}

\noindent\textbf{Case $\alpha\neq a$:}  This case requires some fairly long, although straightforward, computations, which we present schematically.  Firstly, we compute the  $\eqref{eq:complexsist2}-(\alpha-a)\eqref{eq:complexsist1}$ and  $\eqref{eq:complexsist3}-\eqref{eq:complexsist1}$ to remove the $\Gamma$ and obtain two polynomials in $\mathbb{R}[J,a,b,\alpha,H]$, that must vanish for our solution:
\[
\begin{array}{rcl}
	\mathfrak{P}_1&=& -a b^2 H-8 J^2 a H-4 J a H \alpha+6 J a^2 H+2 J a+a H \alpha^2-a^3H \\
	 \noalign{\medskip}
	 &&+2 a \alpha-2 a^2-2 J b^2 H+2 b^2 H \alpha+2 b^2+8 J^2 H \alpha-2 J H \alpha^2-2 J \alpha, \\
	  \noalign{\bigskip}

	\mathfrak{P}_2&=& -8 J a H+2 a H \alpha+a^2 H+2 a+b^2 H+12 J^2 H-4 J H \alpha-3 J+\alpha.
%
%
	  
\end{array}
\]	  
Now, use $\nh(H)=2h(1-H(\lambda-2J))$, $\lambda=6J-2a-\alpha$ and the  derivatives given by Lemma~\ref{le:complex-derivatives} to compute two new polynomials $\mathfrak{P}_3=\frac{\nh(\mathfrak{P}_1)}{h}$ and  $\mathfrak{P}_4\frac{\nh(\mathfrak{P}_2)}{2h}$:
\[
\begin{array}{rcl}
\mathfrak{P}_3&= & 8 J a b^2 H-11 a b^2 H \alpha+4 a^2 b^2 H-22 a b^2-96 J^3 a H-64 J^2 a H \alpha  \\
\noalign{\medskip}
&&+108 J^2 a^2 H+24 J^2 a+6 J a^2 H \alpha+28 J a H \alpha^2-40 J a^3 H+30 J a \alpha \\
\noalign{\medskip}
&&-34 J a^2+a^3 H \alpha-3 a^2 H \alpha^2-3 a H \alpha^3+5 a^4 H-6 a^2 \alpha+4 a \alpha^2 \\
\noalign{\medskip}
&& +10 a^3-36 J^2 b^2 H+30 J b^2 H \alpha+30 J b^2-3 b^2 H \alpha^2-b^4 H+2 b^2 \alpha\\
\noalign{\medskip}
&&+96 J^3 H \alpha-44 J^2 H \alpha^2-24 J^2 \alpha+6 J H \alpha^3+4 J \alpha^2,
 \\
	    \noalign{\bigskip}
	  
	  \mathfrak{P}_4&=& -a b^2 H-24 J^2 a H+8 J a H \alpha+8 J a^2 H+6 J a-a H \alpha^2-a^2 H \alpha-a^3 H \\
\noalign{\medskip}
&&-2 a^2+b^2 H \alpha+2 b^2+24 J^3 H-12 J^2 H \alpha-6 J^2+2 J H \alpha^2+3 J \alpha-\alpha^2.
	 
\end{array}
\]
Finally, we compute $\mathfrak{P}_5=\frac{\nh(\mathfrak{P}_3)}{h}$ and  $\mathfrak{P}_6=\frac{\nh(\mathfrak{P}_4)}{h}$, which gives two additional polynomials:
\[
\begin{array}{rcl}
	  \mathfrak{P}_5&=&676 J^2 a b^2 H-514 J a b^2 H \alpha-92 J a^2 b^2 H-900 J a b^2+94 a^2 b^2 H \alpha  \\
\noalign{\medskip}
&&+51 a b^2 H \alpha^2-20 a^3 b^2 H+20 a b^4 H+12 a b^2 \alpha+280 a^2 b^2-1824 J^4 a H \\
\noalign{\medskip}
&&-1536 J^3 a H \alpha+2744 J^3 a^2 H+456 J^3 a+240 J^2 a^2 H \alpha+1012 J^2 a H \alpha^2 \\
\noalign{\medskip}
&&-1564 J^2 a^3 H+692 J^2 a \alpha-840 J^2 a^2+70 J a^3 H \alpha-234 J a^2 H \alpha^2 \\
\noalign{\medskip}
&&-210 J a H \alpha^3+404 J a^4 H-246 J a^2 \alpha-202 J a \alpha^2+460 J a^3-17 a^4 H \alpha \\
\noalign{\medskip}
&&+15 a^3 H \alpha^2+27 a^2 H \alpha^3+15 a H \alpha^4-40 a^5 H+20 a^3 \alpha+46 a^2 \alpha^2\\
\noalign{\medskip}
&& +14 a \alpha^3-80 a^4-840 J^3 b^2 H+696 J^2 b^2 H \alpha+672 J^2 b^2-138 J b^2 H \alpha^2\\
\noalign{\medskip}
&&-16 J b^4 H+38 J b^2 \alpha+9 b^2 H \alpha^3-9 b^4 H \alpha-18 b^2 \alpha^2-24 b^4+1824 J^4 H \alpha \\
\noalign{\medskip}
&&-1208 J^3 H \alpha^2-456 J^3 \alpha+312 J^2 H \alpha^3+148 J^2 \alpha^2-30 J H \alpha^4-12 J \alpha^3,	  
	  
\\
	  \noalign{\bigskip}
	  
	\mathfrak{P}_6&=& -7 a b^2 H \alpha+4 a^2 b^2 H-22 a b^2-336 J^3 a H+160 J^2 a H+184 J^2 a^2 H \alpha \\
\noalign{\medskip}
&&+84 J^2 a-36 J a H \alpha^2-48 J a^2 H \alpha-48 J a^3 H-12 J a \alpha-50 J a^2\\
\noalign{\medskip}
&&+3 a H \alpha^3+5 a^2 H \alpha^2+5 a^3 H \alpha+5 a^4 H+2 a \alpha^2+2 a^2 \alpha+10 a^3 \\
\noalign{\medskip}
&&-24 J^2 b^2 H+24 J b^2 H \alpha+38 J b^2-3 b^2 H \alpha^2-b^4 H-2 b^2 \alpha+240 J^4 H \\
\noalign{\medskip}
&&-168 J^3 H \alpha-60 J^3+52 J^2 H \alpha^2+42 J^2 \alpha-6 J H \alpha^3-22 J \alpha^2+4 \alpha^3.  
	  
\end{array}
\]
Thus, we have the system of polynomial equations in the variables $J,a,b,\alpha,H$ given by $\{\mathfrak{P}_i=0\}$. Let $\mathcal{I}=\langle\mathfrak{P}_i\rangle$ be the ideal generated by the polynomials $\mathfrak{P}_i$. Notice that a solution of the system is a solution of any polynomial in the ideal $\mathcal{I}\subset \mathbb{R}[J,a,b,\alpha,H]$. Now, we look for an appropriate polynomial by computing a Gröbner basis $\mathcal{G}$ for $\mathcal{I}$ using graded lexicographic order (we refer to \cite{Cox-Little-Oshea} for details on the properties of Gröbner basis and some algorithms used to compute them). As a result, we obtain a basis with 13 polynomials, which include the following:
\[
	\mathfrak{G}=16 b^8+8b^6\alpha^2+b^4\alpha^4\in \mathcal{G}.
\]
Since $\mathfrak{G}\in \mathcal{I}$, it must vanish, so we conclude that  $b=0$ necessarily, which contradicts the assumption that the solution is of Type I.b. \qed

\section{Non-diagonalizable cases with real eigenvalues}\label{sec:non-diagonalizable-real}

In this section, we focus on the non-diagonalizable cases with real eigenvalues. Hence, we will tackle 4-dimensional solutions $(M,g,h)$ of the vacuum weighted Einstein field equations  with $\nh$ spacelike, such that the Ricci operator is of Type II or Type III, as given by \eqref{eq:matrices-type-IIandIII}. We also include in this section any isotropic solutions that are either 2-step nilpotent or 3-step nilpotent, and we will use the same frame as in \eqref{eq:matrices-iso} in this case. As in the previous two sections, all solutions are assumed to have harmonic curvature and the results will complete the proof of Theorem~\ref{th:classification-harmonic}, which is included at the end of the section.

\subsection{Type II solutions}
We begin by analyzing solutions with Ricci operator of Type II. We already know from \cite{Brozos-Mojon} that isotropic solutions have nilpotent Ricci operator. Hence, we consider the case with $\nh$ spacelike first. 

Assume $\operatorname{Ric}$ is of Type II with $\nh$ spacelike. Then, there exists a pseudo-orthonormal frame $\mathcal{B}_2$ as in \eqref{eq:matrices-type-IIandIII}, so that the Ricci operator is given by $\Ric \nh=\lambda \nh$, $\Ric U=\alpha U+\varepsilon V$, $\Ric V=\alpha V$ and $\Ric E_2=\beta E_2$. Now, from the weighted Einstein field equations \eqref{eq:vacuum-Einstein-field-equations}, it follows that the Hessian operator $\hes_h=\nabla \nh$ has the following form in this frame: 
\[
	\hes_h \nh=\tilde{\lambda}\nh, \quad \hes_hU=\tilde{\alpha} U+\varepsilon  h V, \quad  \hes_hV=\tilde{\alpha} V, \quad  \hes_hE_2=\tilde{\beta} E_2,
\] 
where $\tilde{\lambda}=h(\lambda-2J)$, $\tilde{\alpha}=h(\alpha-2J)$ and $\tilde{\beta}=h(\beta-2J)$. 

\begin{lemma}\label{le:split-typeII-spacelike}
	Let $(M,g,h)$ be a Type II solution with $\nh$ spacelike. Then $M$ splits as a direct product $I\times N$, with metric $g^M=dt^2\oplus g^N$ where $\partial_t=E_1=\nh/|\nh|$.
\end{lemma}
\begin{proof}
Since $\hes_hE_1=\tilde{\lambda}E_1$, the distribution generated by $\nh$ is totally geodesic. Furthermore, we see that
\[
\begin{array}{rcl}
	0&=&U(g(\nh,V))=g(\nabla_{U}\nh,V)+g(\nh,\nabla_{U}V)=\tilde{\alpha}+g(\nh,\nabla_{U}V), \\
	\noalign{\medskip}
	0&=&V(g(\nh,U))=g(\nabla_{V}\nh,U)+g(\nh,\nabla_{V}U)=\tilde{\alpha}+g(\nh,\nabla_{V}U),
\end{array}
\]
so $g([U,V],\nh)=0$. Similarly, we verify that $g([U,E_2],\nh)=g([V,E_2],\nh)=0$. Thus, $\mathrm{span}\{U,V,E_2\}$ is close under the Lie bracket, and the distribution generated by $\mathrm{span}\{U,V,E_2\}$ is integrable. Following the same argument that was used in Section~\ref{sec:4-dim-diagonal} to obtain \eqref{eq:decomposition}, it follows that, locally, $M$ splits as a product $I\times N$, with the metric $g^M=dt^2\oplus g^N$, where
 $t$ is  such that $dt=dh/|\nh|$ and $h=h(t)$. Note that, since $\nh$ is spacelike, $\partial_t=\nabla t=E_1$.
\end{proof}

Notice that, from Lemma~\ref{le:split-typeII-spacelike} we can compute $\Hes_{h}(\partial_t,\partial_t)=h''$, and $\lambda=h^{-1}h''+2J$, so $\lambda$ depends only on $t$. This is indeed true for the three eigenvalues.

\begin{lemma}\label{lemma:caseII-coordinate-s-eigenvalues}
	Let $(M,g,h)$ be a Type II solution with $\nh$ spacelike. Then, all eigenvalues of the Ricci operator depend only on the local coordinate $t$. Thus, for the adapted frame $\mathcal{B}_2=\{E_1,U,V,E_2\}$ one has
	\[U(\alpha)=V(\alpha)=E_2(\alpha)=0 \text{ and } U(\beta)=V(\beta)=E_2(\beta)=0.\]
\end{lemma}
\begin{proof}
Since the curvature tensor is harmonic, the Ricci tensor is Codazzi. Hence, on the one hand, we have
\[
\begin{array}{rcl}
 	(\nabla_{\nh}\rho)(U,V)&=&\nh(\rho(U,V))-\rho(\nabla_{\nh}U,V)-\rho(U,\nabla_{\nh}V) \\
	\noalign{\medskip}
	&=&\nh(\alpha)-\alpha g(\nabla_{\nh}U,V)-\alpha g(U,\nabla_{\nh}V)-\varepsilon g(V,\nabla_{\nh}V)\\
	\noalign{\medskip}
	&=&\nh(\alpha),
\end{array}
\]
where we have used $\nh (g(U,V))=0$ and $\nh (g(V,V))=0$. On the other hand,
\[
\begin{array}{rcl}
 	(\nabla_{U}\rho)(\nh,V)&=&-\rho(\nabla_{U}\nh,V)-\rho(\nh,\nabla_{U}V) \\
	\noalign{\medskip}
	&=&-\alpha g(\nabla_{U}\nh,V)-\lambda g(\nh,\nabla_{U}V)=(\lambda-\alpha)\tilde{\alpha},
\end{array}
\]
so we end up with $\nh(\alpha)=h(\lambda-\alpha)(\alpha-2J)$. We can also write $(\nabla_{\nh}\rho)(E_2,E_2)=\nh(\beta)$, and $(\nabla_{E_2}\rho)(\nh,E_2)=(\lambda-\beta)\tilde{\beta}$, so that $\nh(\beta)=h(\lambda-\beta)(\beta-2J)$. Now, since $\tau=\lambda+2\alpha+\beta$ is constant and $\lambda$ depends only on $t$, we apply the same process as in Lemma~\ref{lemma:coordinate-s-eigenvalues} to show that both $\alpha$ and $\beta$ depend only on $t$.
\end{proof}

The fact that the Ricci tensor is Codazzi, together with the information already obtained, allow to compute some Christoffel symbols as follows.

\begin{lemma}\label{lemma:connection-coeffs-typeII}
	Let $(M,g,h)$ be a Type II solution with $\nh$ spacelike. Then, for the adapted frame $\mathcal{B}_2=\{\nh,U,V,E_2\}$, the following equations are satisfied:
	\[
	\begin{array}{rcl}
		0&=&(\alpha-\beta)g(\nabla_{E_2}V,E_2), \\
		\noalign{\medskip}
 		0&=&(\alpha-\beta)g(\nabla_{\nh}U,E_2)+\varepsilon g(\nabla_{\nh}V,E_2), \\
		\noalign{\medskip}
		0&=&(\alpha-\beta)g(\nabla_{\nh}V,E_2), \\
		\noalign{\medskip}
		0&=&g(\nabla_VV,U), \\
		\noalign{\medskip}
		0&=&(\alpha-\beta)g(\nabla_VV,E_2), \\
		\noalign{\medskip}
		0&=& (\alpha-\beta)g(E_2,\nabla_UU)+\varepsilon g(E_2,\nabla_UV)-2\varepsilon g(U,\nabla_{E_2}V), \\
		\noalign{\medskip}
		0&=& (\alpha-\beta)g(\nabla_VU,E_2)+\varepsilon g(\nabla_VV,E_2),\\ 
		\noalign{\medskip}
		0&=& (\alpha-\beta)g(\nabla_UV,E_2), \\
		\noalign{\medskip}
		0&=&(\alpha-\beta)g(\nabla_{E_2}U,E_2)+\varepsilon g(\nabla_{E_2}V,E_2).
	\end{array}
	\]
\end{lemma}
\begin{proof}
	Using Lemma~\ref{lemma:caseII-coordinate-s-eigenvalues} and the Codazzi character of the Ricci tensor, we further analyze the connection coefficients for the different vectors. For example,
\[
 	(\nabla_{V}\rho)(E_2,E_2)=V(\beta)-2\rho(\nabla_{V}E_2,E_2)=-2\beta g(\nabla_{V}E_2,E_2)=0,
\]
and similarly, $(\nabla_{E_2}\rho)(V,E_2)=(\alpha-\beta)g(\nabla_{E_2}V,E_2)$, hence $(\alpha-\beta)g(\nabla_{E_2}V,E_2)=0$. The remaining components of the covariant derivative of the Ricci tensor are computed in a similar manner, and we omit details.
\end{proof}

Once we have obtained enough information on the Levi-Civita connection with respect to the adapted frame, we are ready to give the following classification result for Type II solutions, where we distinguish the isotropic case from that in which $\nh$ is spacelike. 
	
\begin{theorem}\label{th:typeII}
Let $(M,g,h)$ be a $4$-dimensional solution of the vacuum weighted Einstein field equations \eqref{eq:vacuum-Einstein-field-equations} with harmonic curvature and Ricci operator of Type II. 
\begin{enumerate}
	\item If $g(\nh,\nh)> 0$, then $(M,g)$ is a Kundt spacetime. 
\item If $g(\nh,\nh)=0$, then $\operatorname{Ric}$ is $2$-step nilpotent and $(M,g)$ is a $pp$-wave. Moreover, there exist local coordinates $\{u,v,x_1,x_2\}$ such that
			\[
			g_{ppw}(u,v,x_1,x_2)=2\, du\, dv+F(v,x_1,x_2)\,dv^2+dx_1^2+dx_2^2
			\]
			with $h=h(v)$ and $\Delta_xF=\partial^2_{x_1}F+\partial^2_{x_2}F=\frac{-2h''(v)}{h(v)}$.
		\end{enumerate}
	\end{theorem}
	\begin{proof}
Assume first that $\nh$ is spacelike. We work in the adapted pseudo-orthonor\-mal frame $\mathcal{B}_2$ so that $\Ric$ is given by \eqref{eq:matrices-type-IIandIII}. From Lemma~\ref{lemma:connection-coeffs-typeII}, either if $\alpha=\beta$ or $\alpha \neq \beta$, we have
\[
g(\nabla_VV,U)=0, \quad g(\nabla_VV,E_2)=0.
\]
Moreover, $g(\nabla_VV,V)=0$ (since $g(V,V)=0$), and $g(\nabla_V V,\nh)=-\Hes_h(V,V)=0$. Hence $\nabla_VV=0$ and $V$ is geodesic. 

In what follows we show that $V$ satisfies $\nabla_XV=\omega(X)V$ for some 1-form $\omega$ and $X\perp V$ as in \eqref{eq:condition-K1} to show that the spacetime is Kundt. 
We check directly from Lemma~\ref{lemma:connection-coeffs-typeII} that $ g(\nabla_{E_2}V,E_2)=0$ and $g(\nabla_{\nh}V,E_2)=0$. Furthermore, from the structure of $\Hes_h$, it follows that $g(\nabla_{E_2}V,\nh)=-\Hes_h(V,E_2)=0$, $g(\nabla_{\nh}V,\nh)=-\Hes_h(V,\nh)=0$. Finally, since $V$ is lightlike, $g(\nabla_{\nh}V,V)=g(\nabla_{E_2}V,V)=0$. Hence, for any $X\perp V$, we can write $\nabla_XV=\omega(X)V$ for some 1-form $\omega$ satisfying $\omega(V)=0$, $\omega(\nh)=g(\nabla_{\nh}V,U)$ and $\omega(E_2)=g(\nabla_{E_2}V,U)$.

Now, we consider the isotropic case, i.e. assume $\nh$ is lightlike.	
We consider the pseudo-orthonormal frame $\mathcal{B}_0=\{\nh, U, X_1,X_2\}$ so that $\Ric$ takes the form of \eqref{eq:matrices-iso} with $\mu=0$. Since $\Ric$ is  $2$-step nilpotent, the image of every vector field in the basis vanishes except for $\operatorname{Ric}(U)=\nu \nh$. Moreover, equations \eqref{eq:vacuum-Einstein-field-equations} reduce to $h \rho=\Hes_h$, so we have that $\nabla_{\nh} \nh=\nabla_{X_1}\nh=\nabla_{X_2}\nh=0$, and that $\nabla_U \nh=h\nu \nh$. Therefore, $\nh$ is a recurrent vector field.
		
		We compute
		\[
		\begin{array}{rcl}
			0=\rho(X_1,U)&=&R(X_1,U,U,\nabla h)+R(X_1,X_2,U,X_2), \text{ and}\\
			\noalign{\medskip}
			0=\rho(X_1,X_1)&=&2R(X_1,U,X_1,\nabla h)+R(X_1,X_2,X_1,X_2).
		\end{array}
		\]	
		Since $dP=0$ and $J=0$, we obtain from \eqref{eq:Rnf} that $	R(\nh,X,Y,Z)=\rho\wedge dh(X,Y,Z)$, so
		\[
		\begin{array}{rcl}
			R(\nh,U,U,X_1)&=&\rho\wedge dh(U,U,X_1)=0, \text{ and}\\
			R(\nh,X_1,U,X_1)&=&\rho\wedge dh(X_1,U,X_1)=0.
		\end{array}
		\]
		Hence $R(X_1,X_2,X_2,U)=0$ and $R(X_1,X_2,X_1,X_2)=0$. Analogously, we obtain that $R(X_1,X_2,X_1,U)=0$. Therefore, since $\operatorname{Im}(\operatorname{Ric})$ is isotropic and $R(\mathcal{D}^\perp,\mathcal{D}^\perp)=0$ we conclude that that $(M,g)$ is a $pp$-wave (see \cite{leistner}).
		
		Adopt canonical coordinates for a $pp$-wave so that the metric is given as in \eqref{eq:local-coord-pr-xeral} with $\partial_uF=0$. Then the curvature tensor is harmonic if and only if $\Delta_xF=\partial^2_{x_1}F+\partial^2_{x_2}F=\lambda(v)$ is a function of the coordinate $v$. Moreover, a direct computation shows that the only possibly non-vanishing component of $G^h$ is 
		\[
		G^h(\partial_v,\partial_v)=\frac12\left(-2h''(v) - h(v)\left(\partial^2_{x_1}F+\partial^2_{x_2}F\right)\right).
		\] 
		Hence, from $G^h(\partial_v,\partial_v)=0$, we obtain $\partial^2_{x_1}F+\partial^2_{x_2}F=\frac{-2h''(v)}{h(v)}$.
	\end{proof}

\begin{remark}
	Note, from Theorem~\ref{th:typeII}, that Type II solutions of the vacuum weighted Einstein field equations  with $\nh$ spacelike are Kundt spacetimes where the distinguished lightlike vector field is $V$ in the adapted frame $\mathcal{B}_2=\{\nh,U,V,E_2\}$. Indeed, the covariant derivative of $V$ satisfies
	\[
	\nabla_VV=0, \qquad \nabla_{\nh}V=g(\nabla_{\nh}V,U)V, \qquad \nabla_{E_2}V=g(\nabla_{E_2}V,U)V.
	\]
	Associated to any Kundt spacetime, there exist canonical local coordinates as in \eqref{eq:local-coord-kundt-xeral}. However, not every Kundt spacetime has Ricci operator of Type II.
	
	If $\alpha\neq \beta$, using the conditions of the previous results for Type II solutions, we can obtain more specialized coordinates as follows. From the relations obtained in Lemmas~\ref{le:split-typeII-spacelike}, \ref{lemma:caseII-coordinate-s-eigenvalues} and \ref{lemma:connection-coeffs-typeII} we get that 
	\[
	 \nabla_{\nh}\nh=\tilde{\lambda} \nh, \quad \nabla_{E_2}\nh=\tilde{\beta} E_2, \quad \nabla_{E_2}E_2\parallel \nh,   \text{ and } \nabla_{\nh} E_2=0. 
	\]
	Hence, the distribution $\spanned\{\nh,E_2\}$ is totally geodesic and $\spanned\{U,V\}$ is an integrable distribution, so the splitting in Lemma~\ref{le:split-typeII-spacelike} can be further specialized. Thus, there exist local coordinates $\{t,e_2,u,v\}$ so that $h=h(t)$ and the metric takes the form 
	\[
	g(t,e_2,u,v)= dt^2 + r(t,e_2) de_2^2+2 H(t,e_2,u,v) dv du+ F(t,e_2,u,v) dv^2.
	\]
Working with these local coordinates, a direct computation of the Hessian operator of $h$ shows that the eigenvalues are: $h''$, $\frac{h' \partial_{t}r}{2r}$, and $\frac{h' \partial_t H}{2H}$.
Since the Ricci eigenvalues only depend on the coordinate $s$, so do the eigenvalues of $\hes_h$, which are related by equation \eqref{eq:vacuum-Einstein-field-equations}. Hence, $r$ and $H$ decompose as $r(t,e_2)=r_1(t) r_2(e_2)$ and $H(t,e_2,v,u)=H_0(t) H_1(e_2,v,u)$. Moreover, a direct computation of $G^h$ yields 
\[
G^h(\partial_u,\partial_{e_2})=\frac12 h \left(\partial_{e_2} H_1 \partial_{u} H_1-H_1\partial_{e_2}\partial_{u} H_1\right)/H_1^2.
\]
From where $\partial_{e_2} H_1 \partial_{u} H_1-H_1\partial_{e_2}\partial_{u} H_1=0$, which induces an extra decomposition on the function $H_1$ of the form $H_1(e_2,u,v)=H_2(e_2,v) H_3(u,v)$. Hence, the metric can be written as
	\begin{equation}\label{eq:typeII-remark}
	\begin{array}{rcl}
			g(t,e_2,u,v)&=& dt^2 + r_1(t) r_2(e_2) de_2^2+ F(t,e_2,u,v) dv^2 \\
			\noalign{\medskip}
			&&+ 2 H_0(t)H_2(e_2,v) H_3(u,v) du dv.
	\end{array}
	\end{equation}
Generically, metrics given by \eqref{eq:typeII-remark} have Ricci operator of type II. However, they do not have harmonic Weyl tensor nor do they satisfy the vacuum weighted Einstein field equations  in general.  
\end{remark}

\begin{remark}
	Solutions of dimension 4 realized on the family of pure radiation waves, with the metric given by \eqref{eq:local-coord-pr-xeral}, were studied by the authors in \cite{Brozos-Mojon-prwaves}, and classification results were given both in the general case and for solutions with harmonic curvature. In the latter case, it was proved that isotropic solutions are realized on $pp$-waves and non-isotropic ones, on plane waves. In both instances, they are of Type II. This contrasts with the broader family of Kundt spacetimes, where we can find Type III solutions with harmonic curvature tensor (see Example~\ref{ex:iso-3-step-solution}).
\end{remark}

\subsection{Type III solutions}
In this subsection we consider solutions with Ricci operator of Type III as specified in \eqref{eq:matrices-type-IIandIII} or in \eqref{eq:matrices-iso}, depending on the whether $\nh$ is spacelike or lightlike. We already know that there are no solutions of this kind with $\nh$ timelike. In both possible cases, all solutions are realized on Kundt spacetimes as shown in the following result.

\begin{theorem}\label{th:typeIII}
	Let $(M,g,h)$ be a $4$-dimensional solution of the vacuum weighted Einstein field equations \eqref{eq:vacuum-Einstein-field-equations} with harmonic curvature and Ricci operator of Type III. Then:
	\begin{enumerate} 
	\item If $g(\nh,\nh)> 0$, $(M,g)$ is a Kundt spacetime. Moreover, there exist local coordinates as in \eqref{eq:local-coord-kundt-xeral} where $h=h(v,x_1,x_2)$.
	\item If $g(\nh,\nh)=0$ on an open subset, then $(M,g)$ is a Kundt spacetime. Moreover, the Ricci operator is $3$-step nilpotent and there exist local coordinates as in \eqref{eq:local-coord-kundt-xeral} with $h=h(v)$.
\end{enumerate}
\end{theorem}
\begin{proof}
We assume $(M,g,h)$ is a $4$-dimensional solution with harmonic curvature and Ricci operator of Type III. We assume first that $g(\nh,\nh)> 0$.  
According to \eqref{eq:matrices-type-IIandIII} and taking into account Lemma~\ref{lemma:nh-Ricci-eigen}, there exists a suitable adapted frame  $\mathcal{B}_2=\{\nh,U,V,E_2\}$ on which the Ricci operator is given by $\Ric \nh=\lambda \nh$, $\Ric U=\alpha U$, $\Ric V=\alpha V+E_2$ and $\Ric E_2=\alpha E_2+U$. The treatment of solutions of this type is  similar to that of the previous case. However, proving that the gradient of the Ricci eigenvalues has no component in $\nh^\perp$ is simpler. Indeed, we have
\[
\begin{array}{rcl}
	(\nabla_{\nh} \rho)(U,\nh)&=&\nh(\rho(U,\nh))-\rho(\nabla_{\nh}U,\nh)-\rho(U,\nabla_{\nh}\nh) \\
	\noalign{\smallskip}
	&=&-\lambda g(\nabla_{\nh}U,\nh)-\alpha g(U,\nabla_{\nh}\nh)=0,
\end{array}
\]
where we have used $\nabla_{\nh}\nh=h(\lambda-2J)\nh$ and $g(U,\nh)=0$. Now, we can write
\[
\begin{array}{rcl}
	(\nabla_{U} \rho)(\nh,\nh)&=&U(\rho(\nh,\nh))-2\rho(\nabla_{U}\nh,\nh) \\
	\noalign{\smallskip}
	&=&U(\lambda)|\nh|^2+\lambda U(g(\nh,\nh)) -2 \lambda g(\nabla_{U}\nh,\nh)\\
	\noalign{\smallskip}
	&=&U(\lambda)|\nh|^2.
\end{array}
\]
Since the Ricci tensor is Codazzi and $|\nh|^2>0$, it follows that $U(\lambda)=0$. Moreover, $\tau=\lambda+3\alpha$ is constant, so $U(\alpha)=0$. We can similarly compute the analogous covariant derivatives for $V$ and $E_2$ instead of $U$, yielding $V(\lambda)=V(\alpha)=0$ and $E_2(\lambda)=E_2(\alpha)=0$.

%

In order to show that $(M,g)$ is a Kundt spacetime, we are going to see that the lightlike vector $U$ satisfies
	\[
		\nabla_UU=0, \qquad \nabla_{\nh}U=g(\nabla_{\nh}U,V)U, \qquad \nabla_{E_2}U=g(\nabla_{E_2}U,V)U.
	\]
	The process of computing the covariant derivatives of the Ricci tensor and applying the Codazzi condition is the same as in previous instances, so we omit details. First, we compute $(\nabla_{U} \rho)(U,V)=-g(\nabla_{U}U,E_2)$ and $(\nabla_{V} \rho)(U,U)=0$, so $g(\nabla_{U}U,E_2)=0$. Similarly, we have $(\nabla_{U} \rho)(E_2,V)=g(\nabla_{U}U,V)$ and $(\nabla_{V} \rho)(E_2,U)=0$, so $g(\nabla_{U}U,V)=0$. The component in the direction of $\nh$ is easier to compute: $g(\nabla_UU,\nh)=-\Hes_h(U,U)=0$. Similarly, since $U$ is lightlike, it is immediate that $g(\nabla_UU,U)=0$. Thus, we have proved that $\nabla_UU=0$.

Now, we need to compute $\nabla_{\nh}U$ and $\nabla_{E_2}U$. For the first derivative, we can write $g(\nabla_{\nh}U,U)=0$ and $g(\nabla_{\nh}U,\nh)=-\Hes_h(U,\nh)=0$. Therefore, we only need to determine the component given by  $g(\nabla_{\nh}U,E_2)$.  To that end, consider the covariant derivatives $(\nabla_{\nh} \rho)(U,V)=\nh(\alpha)-g(\nabla_{\nh}U,E_2)$ and $(\nabla_{V} \rho)(\nh,U)=h(\alpha-2J)(\lambda-\alpha)$. Moreover, $h(\alpha-2J)=-\frac{1}3h\lambda$ so, by the Codazzi condition on the Ricci tensor, we have
 \begin{equation}\label{eq:nlambda1}
 	\nabla h(\alpha)=g(\nabla_{\nh}U,E_2)+\frac{h}3(\alpha-\lambda)\lambda.
\end{equation}
Alternatively, we can compute $(\nabla_{\nh} \rho)(E_2,E_2)=\nh(\alpha)+2g(E_2,\nabla_{\nh}U)$ and $(\nabla_{E_2} \rho)(\nh,E_2)=\frac{h}3(\lambda-\alpha)\lambda$, so
 \begin{equation}\label{eq:nlambda2}
	\nabla h(\alpha)=-2g(\nabla_{\nh}U,E_2)+\frac{h}3(\alpha-\lambda)\lambda.
\end{equation}
Combining \eqref{eq:nlambda1} and \eqref{eq:nlambda2}, it follows that $g(\nabla_{\nh}U,E_2)=0$.

Finally, for the derivative $\nabla_{E_2}U$, since $g(\nabla_{E_2}U,\nh)=-\Hes_h(E_2,U)=0$, we only need to compute $g(\nabla_{E_2}U,E_2)$. To that end, we use the fact that $(\nabla_{E_2}\rho)(V,U)=g(E_2,\nabla_{E_2}U)=(\nabla_{V}\rho)(E_2,U)=0$.

Thus, we can write $\nabla_{\nh}U=g(\nabla_{\nh}U,V)U$ and $\nabla_{E_2}U=g(\nabla_{E_2}U,V)U$. In summary, we have $\nabla_UU=0$ and that, for every $X\perp U$, $\nabla U=\omega\otimes U$, for the 1-form $\omega$ defined in $U^\perp$, and given by $\omega(U)=0$, $\omega(\nh)=g(\nabla_{\nh}U,V)$ and $\omega(E_2)=g(\nabla_{E_2}U,V)$. By the characterization given by \eqref{eq:condition-K1}, the underlying manifold $(M,g)$ is a Kundt spacetime where $U$ is the distinguished lightlike vector field.

Additionally, in local coordinates \eqref{eq:local-coord-kundt-xeral}, the distinguished lightlike geodesic vector field is $\partial_u$, which is orthogonal to $\nh$. Hence $h=h(v,x_1,x_2)$ and Theorem~\ref{th:typeIII}~(1) follows.

Now, if $g(\nh,\nh)=0$ on an open subset, then the Kundt character and the nilpotency of $\Ric$ follow from \cite{Brozos-Mojon}. In local coordinates as in \eqref{eq:local-coord-kundt-xeral}, since the distinguished geodesic lightlike vector field in these coordinates is $\partial_u$, we have that $\nh\parallel \partial_u$, so a direct computation yields $h(u,v,x_1,x_2)=h(v)$.
\end{proof}

\begin{example}\label{ex:iso-3-step-solution}
The structure of isotropic solutions with $3$-step nilpotent Ricci operator is very rigid. However, we can build examples for any arbitrary nowhere constant density function $h=h(v)$. Consider the Kundt metric given by
\[
g=dv\left(2du+F(u,v,x)dv+\omega(u,v,x)dx_1\right)+g(v,x)(dx_1^2+dx_2^2),
\]
with $F(u,v,x_1,x_2)=\frac{u^2 h(v)^4}{C x_1^2}-\frac{12 C x_1^2 (\log (x_1)-1) h'(v)^2}{h(v)^6}$, $\omega(u,v,x_1,x_2)=-\frac{3 C x_1 h'(v)}{h(v)^5}-\frac{2 u}{x_1}$ and $g(v,x_1,x_2)=\frac{C}{h(v)^4}$, $C\neq 0$.
This is a solution for the vacuum weighted Einstein field equations  where
$\Ric$ has $3$-step nilpotent Ricci operator and harmonic Weyl tensor, but it is not locally conformally flat since the Weyl tensor does not vanish (for example, $W(\partial_u,\partial_v,\partial_v,\partial_{x_1})=\frac{3 h'(v)}{2 x_1 h(v)}\neq 0$).

\end{example}

After analyzing solutions of all four admissible types, we are finally in a position to prove the complete statement of Theorem~\ref{th:classification-harmonic}.

\subsection{Proof of Theorem~\ref{th:classification-harmonic}}

Let $(M,g,h)$ be a $4$-dimensional solution of the vacuum weighted Einstein field equations~\eqref{eq:vacuum-Einstein-field-equations} such that $(M,g)$ has harmonic curvature tensor (not locally conformally flat). Additionally, assume that $\Ric$ does not change type in $M$. 

For Type I.a (diagonalizable) solutions, applying Theorem~\ref{th:classification-diagonalizable} in the non-isotropic case gives Theorem~\ref{th:classification-harmonic}~(1), whereas Ricci-flat isotropic solutions fall into Theorem~\ref{th:classification-harmonic}~(2)(a). 

In the non-diagonalizable case, we first use Theorem~\ref{th:non-existence-complex} to prove that there are no solutions of Type I.b., so all remaining admissible solutions are of Type II or Type III. Hence, we apply Theorems~\ref{th:typeII} and \ref{th:typeIII} to complete Theorem~\ref{th:classification-diagonalizable}~(2)(a) in the isotropic case, and Theorem~\ref{th:classification-diagonalizable}~(2)(b) in the non-isotropic case.


\section{Conclusions}\label{sec:conclusions}

We have studied solutions to the vacuum weighted Einstein field equations~\eqref{eq:vacuum-Einstein-field-equations} under two distinct geometric conditions: local conformal flatness and the less restrictive condition of harmonic curvature tensor. In the latter case, we have focused on 4-dimensional solutions. In general, solutions of equation~\eqref{eq:vacuum-Einstein-field-equations} are realized on families of manifolds typically present in the context of General Relativity (see Theorems~\ref{th:loc-conf-flat-ndim} and \ref{th:classification-harmonic}). The analysis we have carried out illustrates some important aspects:
\begin{itemize}
	\item The vacuum weighted Einstein field equations~\eqref{eq:vacuum-Einstein-field-equations} were already introduced in \cite{Brozos-Mojon} by finding suitable analogues to the characterizing properties of the usual Einstein tensor (see \cite{Lovelock}). However, in this work, we discussed a variational approach that yields the same field equations, giving further motivation for their definition by characterizing solutions as critical points of a restricted variation of the weighted Einstein-Hilbert functional~\eqref{eq:HE-action-weighted}.
	\item Locally conformally flat non-isotropic solutions, as well as 4-dimensional non-isotropic solutions with harmonic curvature and diagonalizable Ricci operator, present similar geometric features as those of vacuum static spaces in Riemannian signature (see \cite{Kobayashi} and \cite{Kim-Shin} respectively). Namely, they are realized on specific warped  or direct products (see Theorems~\ref{th:loc-conf-flat-ndim}~(1) and \ref{th:classification-diagonalizable}) in a way analogous to that of the Riemannian case. However, types of solutions not present in Riemannian signature arise as a consequence of working with a Lorentzian metric. Indeed, isotropic solutions are purely Lorentzian, but non-isotropic solutions with non-diagonalizable Ricci operator also appear. Remarkably, both isotropic and non-diagonalizable, non-isotropic solutions with harmonic curvature are realized on Kundt spacetimes (see Theorem~\ref{th:classification-harmonic} and also Theorems~\ref{th:typeII} and \ref{th:typeIII}).
	\item Although the geometry of isotropic solutions is quite rigid, in the sense that all of them are Kundt spacetimes, different assumptions on the Weyl tensor have an impact on the admissible geometries. For example, locally conformally flat solutions are realized on very specific plane waves (see Theorem~\ref{th:loc-conf-flat-ndim}~(2)). In contrast, the condition of harmonic curvature allows for more flexibility, including $pp$-waves which are not plane waves (see Theorems~\ref{th:typeII} and \ref{th:typeIII}) and Kundt spacetimes of Type III (see Example~\ref{ex:iso-3-step-solution}.
	\item Although both the geometry and the density function are determined in the diagonalizable case (see Theorem~\ref{th:classification-diagonalizable}), there is some flexibility for the choice of the density in other cases. For example, it is possible to build locally conformally flat solutions on plane waves for an arbitrary one-variable density function (Theorem~\ref{th:loc-conf-flat-ndim}). Also, Theorems~\ref{th:typeII} and \ref{th:typeIII} illustrate this freedom for the building of solutions with a prescribed density function in manifolds with harmonic curvature, as it does Example~\ref{ex:iso-3-step-solution}.
\end{itemize}


\appendix

\section{Complement to the proof of Theorem~\ref{th:classification-diagonalizable}}\label{appendix}

This section treats the case with three different eigenvalues corresponding to Section~\ref{sec:4-dim-diagonal}. We adopt the notation in that section and assume that the three eigenvalues $\lambda_2$, $\lambda_3$ and $\lambda_4$ are distinct. We are going to see in Lemma~\ref{lemma:all-eigen-different} below that this case is not admissible. The proof is essentially a Lorentzian analogue to that given in \cite{Kim-Shin}. Let $\varepsilon_1=-\varepsilon_2=\varepsilon$ and $\varepsilon_3=\varepsilon_4=1$. From equation \eqref{eq:connection-coefs}
we know that
\begin{equation}\label{eq:connection-coefs-4dim}
\begin{array}{rcl}
	\nabla_{E_1}E_1&=&0, \quad \nabla_{E_i}E_1=\beta_iE_i, \quad \nabla_{E_i}E_i=-\varepsilon\varepsilon_i\beta_i E_1, \\
	\noalign{\medskip}
	\nabla_{E_1}E_i&=&0, \quad \nabla_{E_i}E_j=\varepsilon_k \Gamma_{ijk}E_k,
\end{array}
\end{equation}
where $\{i,j,k\}=\{2,3,4\}$. 
From these relations, we obtain the expression of the Ricci eigenvalues as follows.

\begin{lemma}\label{lemma:Ricci-components}
	Let $(M,g,h)$ be a 4-dimensional non-isotropic solution of the weighted Einstein field equations  \eqref{eq:vacuum-Einstein-field-equations} with harmonic curvature tensor, such that the Ricci operator diagonalizes in the adapted local frame $\mathcal{B}_1=\{E_1,\dots,E_4\}$ and the eigenvalues $\lambda_2$, $\lambda_3$ and $\lambda_4$ are pairwise distinct. Then, they take the following forms:
	\[
		\begin{array}{rcl}
			-\varepsilon \lambda_2&=&\beta_2^2+\varepsilon \beta_2' +\beta_2\beta_{3}+\beta_2\beta_{4}-2\Gamma_{342}\Gamma_{432},	\\
			\noalign{\medskip}
			-\varepsilon \lambda_3&=&\beta_3^2+\varepsilon \beta_3'+\beta_{3}\beta_2+\beta_3\beta_{4}+2\frac{\beta_2-\beta_4}{\beta_3-\beta_4}\Gamma_{342}\Gamma_{432}, \\
			\noalign{\medskip}
			-\varepsilon \lambda_4&=&\beta_4^2+\varepsilon\beta_4'+\beta_4\beta_{2}+\beta_4\beta_{3}+2\frac{\beta_2-\beta_3}{\beta_4-\beta_3}\Gamma_{342}\Gamma_{432}.
		\end{array}
	\]
\end{lemma}
\begin{proof}
	On the one hand, from \eqref{eq:connection-coefs-4dim} and for $i,j\neq 1$, we compute
	\[
	\begin{array}{rcl}
		\nabla_{[E_j,E_i]}E_i&=&\nabla_{\nabla_{E_j}E_i}E_i-\nabla_{\nabla_{E_i}E_j}E_i =\varepsilon_k \varepsilon_j(\Gamma_{j ik}-\Gamma_{ijk})\Gamma_{kij}E_j, \\
		\noalign{\medskip}
		\nabla_{E_j}\nabla_{E_i}E_i&=&-\varepsilon\varepsilon_i\beta_i \nabla_{E_j}E_1=-\varepsilon\varepsilon_i\beta_i\beta_{j}E_j, \\
		\noalign{\medskip}
		\nabla_{E_i}\nabla_{E_j}E_i&=&\varepsilon_k\nabla_{E_i}(\Gamma_{jik}E_k) =\varepsilon_kE_i(\Gamma_{jik})E_k +\varepsilon_k \varepsilon_j\Gamma_{jik}\Gamma_{i kj}E_j,
	\end{array}
	\]
	while, on the other hand,
	\[
	\begin{array}{rcl}
		\nabla_{[E_1,E_i]}E_1&=&\nabla_{\nabla_{E_1}E_i}E_1-\nabla_{\nabla_{E_i}E_1}E_1 =-\beta_i^2E_i, \\
		\noalign{\medskip}
		\nabla_{E_1}\nabla_{E_i}E_1&=&\nabla_{E_1}(\beta_iE_i)=\varepsilon \beta_i'E_i, \\
		\noalign{\medskip}
		\nabla_{E_i}\nabla_{E_1}E_1&=&0.
	\end{array}
	\]
	Hence, we have the curvature components given by
	\[
		\begin{array}{rcl}
			R_{j i j i}&=&-\varepsilon\varepsilon_i\varepsilon_j\beta_i\beta_j+\varepsilon_k\{(\Gamma_{ijk}-\Gamma_{j ik})\Gamma_{kij}-\Gamma_{jik}\Gamma_{i kj}\},  \\
			\noalign{\medskip}
			R_{1i1j}&=&-\varepsilon_i(\varepsilon\beta_i'+ \beta_i^2)\delta_{ij}, 
		\end{array}
	\]
	so the components of the Ricci tensor take the form:
	\[
	\begin{array}{rcl}
		-\varepsilon\lambda_2=\rho_{22}&=&\varepsilon R_{1212}+ R_{3232}+R_{4242} \\
		\noalign{\medskip}
		&=&\varepsilon \beta_2'+\beta_2^2+\beta_2\beta_{3}+\beta_2\beta_{4}\\
		\noalign{\medskip}
		&&+(\Gamma_{234}-\Gamma_{324})\Gamma_{423}-\Gamma_{324}\Gamma_{2 43}+(\Gamma_{243}-\Gamma_{4 23})\Gamma_{324}-\Gamma_{423}\Gamma_{2 34} \\
		\noalign{\medskip}
		&=& \beta_2^2+\varepsilon \beta_2' +\beta_2\beta_{3}+\beta_2\beta_{4}-2 \Gamma_{342}\Gamma_{432},
	\end{array}
	\]
	and
	\[
	\begin{array}{rcl}
		\lambda_3=\rho_{33}&=&\varepsilon R_{1313}-\varepsilon R_{2323}+R_{4343} \\
		\noalign{\medskip}
		&=&-\varepsilon \beta_3^2- \beta_3'-\varepsilon(\beta_2\beta_3+\beta_4\beta_3)+ 2\varepsilon\Gamma_{243}\Gamma_{423}   \\
		\noalign{\medskip}
		&=& -\varepsilon (\beta_3^2+\varepsilon \beta_3'+\beta_{3}\beta_2+\beta_3\beta_{4}-2\Gamma_{243}\Gamma_{423}),
	\end{array}
	\]
	where we have used that $\Gamma_{jki}=-\Gamma_{jik}$. The computation of $\lambda_4$ is analogous to the previous one. Now, since the Ricci tensor is Codazzi, so is $\Hes$ by equation~\eqref{eq:vacuum-Einstein-field-equations}. Hence, we get to the relations $(\beta_j-\beta_k)\Gamma_{ijk}=(\beta_i-\beta_k)\Gamma_{jik}$. With these and  $\Gamma_{jki}=-\Gamma_{jik}$, we have 
	\[
	\Gamma_{243}\Gamma_{423}=-\frac{\beta_2-\beta_4}{\beta_3-\beta_4}\Gamma_{342}\Gamma_{432}, \qquad \Gamma_{234}\Gamma_{324}=-\frac{\beta_2-\beta_3}{\beta_4-\beta_3}\Gamma_{342}\Gamma_{432},
	\]
	from where the result follows.
\end{proof}

\begin{lemma}\label{lemma:all-eigen-different}
	Let $(M,g,h)$ be a 4-dimensional non-isotropic solution of the weighted Einstein field equations \eqref{eq:vacuum-Einstein-field-equations} with harmonic curvature tensor, such that the Ricci operator diagonalizes in the adapted local frame $\mathcal{B}_1=\{E_1,\dots,E_4\}$. Then, the eigenvalues $\lambda_2$, $\lambda_3$ and $\lambda_4$ cannot be pairwise distinct.
\end{lemma}

\begin{proof}
	
	In addition to Lemma~\ref{lemma:Ricci-components}, we have two more possibilities in order to express these components of the Ricci tensor. The first one is using the expression $R_{1i1j}=-\varepsilon\varepsilon_i(\lambda_i-2J)\delta_{ij}$ given by \eqref{eq:Rnf}, which yields
	\begin{equation}\label{eq:Ricci-1}
		-\varepsilon\lambda_i=\varepsilon_i R_{i1i1}-2\varepsilon J=-\varepsilon\beta_i'- \beta_i^2- 2\varepsilon J.
	\end{equation}
	The other option is using the weighted Einstein equation  \eqref{eq:vacuum-Einstein-field-equations-2} itself, so we get
	\begin{equation}\label{eq:Ricci-2}
		-\varepsilon \lambda_i=-\varepsilon \frac{h'}{h}\beta_i-2\varepsilon J.
	\end{equation}
	For the sake of clarity, let $\beta_2=a$, $\beta_3=b$, $\beta_4=c$, and let $\Gamma=\Gamma_{342}\Gamma_{432}$. Then, we can take the expressions given both by Lemma~\ref{lemma:Ricci-components} and by \eqref{eq:Ricci-1} for the difference $R_{22}-R_{33}$ and add them to find
	\[
	-2\varepsilon(\lambda_2-\lambda_3)=(a-b)c-2\frac{a+b-2c}{b-c}\Gamma,
	\]
	while taking the expression given by \eqref{eq:Ricci-2} gives 
	\[
	-2\varepsilon(\lambda_2-\lambda_3)=-2\varepsilon (a-b)\frac{h'}{h}.
	\]
	Equating both expressions yields a first value for $h'h^{-1}$,
	\[
	-\varepsilon \frac{h'}{h}=\frac{c}{2}-\frac{a+b-2c}{(a-b)(b-c)}\Gamma.
	\]
	By the same process, using the components $\lambda_2$ and $\lambda_3$, we have another expression for $h'h^{-1}$:
	\[
	-\varepsilon \frac{h'}{h}=\frac{b}{2}-\frac{a+c-2b}{(b-c)(c-a)}\Gamma.
	\]
	We can now use both values to solve for $\Gamma$ in terms of $a$, $b$ and $c$. Indeed, take 
	\[
	\begin{array}{rcl}
		P&=&a^2+b^2+c^2-ac-ab-bc \\
		\noalign{\medskip}
		&=&\frac{1}2(a-b)^2+\frac{1}2(a-c)^2+\frac{1}2(b-c)^2\geq 0,
	\end{array}
	\] 
	with equality only being achievable when $a=b=c$. Then, we have
	\[
	\Gamma=-\frac{(a-b)(a-c)(b-c)^2}{4P}.
	\]
	Now, consider that $(\beta_4-\beta_2)\Gamma_{34}^2=(\beta_3-\beta_2)\Gamma_{43}^2$. Thus, we can write
	\[
	(\Gamma_{432})^2=\Gamma_{432}\Gamma_{432}=\frac{c-a}{b-a}\Gamma_{342}\Gamma_{432}=-\frac{(a-c)^2(b-c)^2}{4P}.
	\]
	It follows that either $a=c$ or $b=c$, which is a contradiction. Thus, $\beta_2$, $\beta_3$ and $\beta_4$ cannot be pairwise distinct, and the same holds for $\lambda_2$, $\lambda_3$ and $\lambda_4$.
\end{proof}

\end{document}